\documentclass[onefignum,onetabnum]{siamart190516}

\usepackage{xr}

\usepackage{lipsum}
\usepackage{amsfonts}
\usepackage{graphicx}
\usepackage{epstopdf}
\usepackage{algorithmic}
\ifpdf
  \DeclareGraphicsExtensions{.eps,.pdf,.png,.jpg}
\else
  \DeclareGraphicsExtensions{.eps}
\fi

\newsiamremark{remark}{Remark}
\newsiamremark{hypothesis}{Hypothesis}
\crefname{hypothesis}{Hypothesis}{Hypotheses}
\newsiamthm{claim}{Claim}

\headers{Value-Gradient of Optimal Control}{A. Bensoussan, J. Han, P. Yam and X. Zhou}

\title{Value-Gradient based Formulation of Optimal Control Problem
and Machine 
Learning Algorithm}

\author{Alain Bensoussan\footnotemark[3] \thanks{Naveen Jindal School of Management, University of Texas at Dallas, Richardson, Texas 75080-3021, U.S.A.}
\and Jiayue Han\thanks{School of Data Science, City University of Hong Kong, Kowloon, Hong Kong SAR.} \and Sheung Chi Phillip Yam\thanks{Department of Statistics, Chinese University of Hong Kong, Shatin, N.T., Hong Kong SAR.}
 \and Xiang Zhou\footnotemark[3] \thanks{Department of Mathematics, City University of Hong Kong, Kowloon, Hong Kong SAR.} }

\usepackage{amsopn}

\ifpdf
\hypersetup{
  pdftitle={Value-Gradient based Formulation of Optimal Control Problem
and Machine 
Learning Algorithm},
  pdfauthor={A. Bensoussan, J. Han, P. Yam and X. Zhou}}
\fi

\theoremstyle{plain}

\newtheorem{assumption}[theorem]{Assumption}

 \newtheorem{example}{Example}

\renewcommand{\theequation}{\arabic{section}.\arabic{equation}}

\usepackage{subfigure} 
\usepackage{multirow}
\usepackage{array}
\usepackage{mathtools}
\newcolumntype{C}[1]{>{\centering}p{#1}}

\newcommand{\set}[1]{\left\{#1\right\}}

\newcommand{\Real}{\mathbb {R}}

\newcommand{\tr}{\textsf{\tiny T}}

\newcommand{\argmin}{ \operatornamewithlimits{argmin} }

\renewcommand{\d}{\ensuremath{\mathrm{d}}}
\newcommand{\dt}{ \ensuremath{\mathrm{d} t } }
\newcommand{\dx}{ \ensuremath{\mathrm{d} x} }

\newcounter{rownumbers}

\usepackage[normalem]{ulem}
\usepackage[colorinlistoftodos,prependcaption,textsize=tiny]{todonotes}

\usepackage{xcolor}
\usepackage{soul}
\usepackage[toc,page]{appendix}

\begin{document}

\maketitle

\begin{abstract}
Optimal control problem is typically solved by first finding the value function through  Hamilton–Jacobi equation (HJE) and then taking  the minimizer  of the Hamiltonian to obtain the control. 
 In this work, instead of focusing on the value function, we propose  a new formulation 
for the gradient of the value function 
(value-gradient) as  a decoupled system of   partial differential equations  
in the context of continuous-time deterministic discounted optimal control problem. 
We develop an efficient iterative scheme for this system of equations in parallel by utilizing the properties that they share the same characteristic curves 
as the HJE  for the value function. 
For the theoretical part, we prove that this iterative scheme converges linearly in $L_\alpha^2$ sense for some suitable exponent $\alpha$ in a weight function. For the numerical method, we combine characteristic line method with machine learning techniques. Specifically, we generate multiple characteristic curves 
  at each policy 
 iteration from an ensemble of  initial states, 
and compute both the value function and its gradient simultaneously on each  curve as the labelled data. 
Then supervised machine learning is applied to  minimize the weighted squared loss for both the value function and its gradients. 
Experimental results demonstrate  that  this new method not only significantly increases the accuracy but also 
improves the efficiency and   robustness of the numerical estimates, particularly
with less amount of characteristics   data or  fewer training steps.

\end{abstract}

\begin{keywords}
Optimal control, value function, Hamilton-Jacobi equation, machine learning, 
characteristic curve.
\end{keywords}

\begin{AMS}
65K05, 93-08
\end{AMS}

\section{Introduction}
It is well known that  the study of  Hamilton-Jacobi equation (HJE) is one of the core topics in optimal control theory for controlling continuous-time differential dynamical systems  by the principle of    dynamical programming \cite{Fleming1975Book,Bertsekas2001,fleming2006controlled,bensoussan2018estimation}. 
This  equation  is a first-order  nonlinear partial differential equation (PDE) for the value function which maps an arbitrary given initial state 
  to the optimal value of the cost function. Once this HJE solution is known, it can be used to construct the optimal
  control by taking the  minimizer  of the Hamiltonian.  Such an optimal control is the feedback control and it does not depend on knowledge of  initial conditions.

Although theoretically well-developed, numerical methods for the problem is yet to be studied. Because only few optimal control problems, such as the linear quadratic problem (LQR) \cite{bensoussan2018estimation}, have analytical solutions. Solving the PDE given by HJE is not easy, even for the LQR case, in which HJE is converted to a Riccati equation. Moreover, since the dimension of the HJE is the dimension $d$ of state variable $x$ in the dynamical system, 
the size of the state-discretized problems in solving HJEs increases exponentially with $d$.
 This ``curse of dimensionality'' has been the long-standing challenge in solving the high dimensional HJEs,
 and recently there have been rapid and abundant developments  to mitigate this challenge 
by combining optimal control algorithms with machine learning algorithms, particularly  reinforcement learning and deep neural networks \cite{sutton2018reinforcement,Bertsekas2019,Recht2019tour,2020handbookXZ}.

In the literature,  there exists an extensive research on various numerical methods of finding the approximate solution to the HJEs. 
   One important idea which attracted a considerable amount of attention is termed  the successive approximation method   \cite{BEARD19972159,Bea1998,Beard:1998ti},
 which aims to handle the nonlinearity in  the HJE.  The successive approximation method reduces the  nonlinear HJE to an iterative sequence of linear PDEs called the  generalized Hamilton-Jacobi equation (GHJE) and the point-wise optimization  of taking the   minimizer  of the Hamiltonian.
 The  GHJE is linear since the feedback control is given from the previous iteration.
  Therefore traditional numerical PDE methods such as  Galerkin spectral method
 (Successive Galerkin approximation \cite{Bea1998}) for small $d$ can be applied to solve these GHJEs.
 If the dimension is moderately large,  various methods based on  low-dimensional representation {\it ansatz} 
 such as polynomial or low-rank tensor product  \cite{7040310,KK2018,Oster2019} 
 usually work in many applications. For very high dimensional settings, the use of deep neural network
 is prevalent. 
  This two-step procedure in the successive approximation shares exactly the same idea as  {\it policy iteration} in the reinforcement learning \cite{sutton2018reinforcement,Bertsekas2019}.
  
 When the Hamiltonian minimization has a closed-form,   the  HJE 
 can be solved directly by   using grids 
  and finite difference discretization, such as 
 the Dijkstra-type methods like level method \cite{OSHER198812}, fast marching \cite{FastMarching1994}, 
  fast sweeping method \cite{FastSweep2003}, 
  and semi-Lagrangian approximation scheme  \cite{falcone2013semi}. But these grid-based methods suffer from the curse of  dimensionality, i.e., they generally scale up exponentially    with increases in dimension in the space.  There have been  tremendous   
  advances in numerical  methods and empirical  tests   now  for high dimensional PDEs 
   by taking advantage of neural networks to represent high dimensional functions.
  For the HJE in deterministic optimal control problems,    various approaches have been proposed
  and most of them are based on certain forms of Lagrangian formulation equivalent to the HJE.
  For example, 
under certain conditions (such as 
convexity) on the Hamiltonian or the terminal cost, the inspiring works  in
\cite{DarbonOsher2016RMS,ChowOsher2017JSC,ChowOshersplitting2018,ChowOsherADMM2018,ChowOsher2019JSC}  rely on  the  generalized Lax and  Hopf  formulas 
to   transform   the computation of the value function at an arbitrarily given  space-time point   
as  an optimization problem for  the terminal value of the Lagrangian multiplier $p$ \footnote{also called co-state or adjoint variable.},  subject to the  characteristics equation of Hamiltonian ordinary differential equation (ODE)
for $(x,p)$.
In a similar but different style, 
\cite{Kang2015causality} worked with the 
Pontryagin’s maximum  principle (PMP)
by  considering  the characteristic equations of   
 the state  $x(t)$ and the co-state  $\lambda(t)$ 
 as  a two-point boundary value problem  (BVP).
 The optimal feedback control, the value function and
the gradient of the value function on the optimal  trajectories
are computed first by solving the BVP numerically.
With the data generated from the BVP on  characteristic trajectories,
the HJE solution is then  interpolated at any   point
by either using sparse grid  interplants \cite{Kang2017} or minimizing the mean square errors
\cite{Izzo01145,Kang2019,Kang2021QRnet}. This step is the standard form of supervised learning, 
and  the numerical accuracy is determined by the quality of the interpolant  and the amount of the training data. 
For a very large $d$, the curse of dimensionality is mitigated  by the supreme power of deep neural network
in deep learning.    For the review of solving 
high dimensional PDE including the HJE,
refer to the recent review paper \cite{2020HighPDEReivewE}.
 
 In the present  paper, 
 we  shall develop a new formulation as an alternative to the HJE for the optimal control theory
 and this formulation   focuses on the gradient of the value function, 
 instead of the value function itself.
 For brevity, we call this vector-valued gradient function
   as {\it  value-gradient} function.
One of our motivations is that   in practical applications, the optimal feedback control or the optimal policy, 
is  the ultimate goal of the decision maker and   this optimal policy is completely determined by  the  value-gradient  in minimizing  the Hamiltonian. 
Another    motivation to investigate this   value-gradient  function comes from the training step 
where we want to provide  the data not only for the value function but also for its gradient
to enhance the accuracy of the interpolation.
  Our new formulation has the following nice properties:
(1) The proposed method has a linear convergence rate.
 (2) It is  a closed system of  PDEs for components of  vector-valued 
  value-gradient functions. 
 (3) This system is essentially decoupled in each component and is perfectly suitable for 
 parallel computing in policy iteration. 
 (4) Each PDE in the system has the exact same characteristics equation as
  the original HJE for the value function.
(5) After simulating   characteristics curves, 
we obtain the results of the value function and the  value-gradient function simultaneously on the characteristics curves
to  train the value function in the whole space.

We demonstrate our novel method by
  focusing on  the infinite-horizon discounted deterministic optimal control problem.
This setup will simplify our presentation since the HJE is stationary in time. In addition, we assume the value functions in concern are sufficiently smooth, at least $C^{2}$,
which can be guaranteed by imposing appropriate conditions on the state dynamics and the running cost functions.  So, we can interpret the  system of PDEs
for  value-gradient functions in the classical sense. 

We develop the numerical algorithm based on  the {\it policy iteration} \cite{sutton2018reinforcement} 
and the method of characteristics \cite{evans_2010}. Under an assumption on the dynamics and the payoff function, we show by mathematical induction that the value-gradient function at each iteration and its corresponding control are uniformly bounded by linearly growth functions, while the gradients of these two functions are uniformly bounded by constants. With \cref{lem:iter} and \cref{lem:bound}, this algorithm is proved to converge linearly in $L_\alpha^2$ sense (see \cref{thm: main}) for some suitable exponent $\alpha$ in a weight function. As for the algorithm, in each policy iteration, only linear equations are solved on the characteristics curves
starting from a collection of initial states.
The interpolation or the training step is to minimize the convex combination of the  mean  squared errors 
of both the value function and the  value-gradient functions.  
One prominent benefit of our algorithm is that
we can combine the data from both value and value-gradient
since they share the same characteristics.
So, the output of our algorithm is still the value function, which is 
approximated by any type of non-parametric functions like 
radial basis functions or   neural networks.
The  value-gradient function is obtained by automatic differentiation. 
Our extensive numerical examples 
confirm that  the accuracy  and the robustness  
are both significantly improved in comparison to 
only  solving the HJE in  the same 
policy iteration method. Finally, we remark that a preliminary idea in this paper has appeared in the authors' recent manuscript  \cite{2020handbookXZ} on review of  machine learning and control theory. Here we present the full development and propose the detailed numerical methods based on machine-learning, with emphasis on theoretical proof of $L_\alpha^2$ convergence.

The paper is organized as follows. 
Section \ref{sec:2} is   the problem setup for the optimal control problem 
and  the review of
HJEs  and  Pontryagin’s maximum principle (PMP) with their connections to the theory of optimal control . 
Section \ref{sec:main} is  our new formulation in terms of the  value-gradient function with the convergence analysis of the iterative scheme.
Section \ref{sec:nm} presents our main algorithms and Section \ref{sec:nex} is our numerical examples. Section \ref{sec:con} includes   some discussions on generalization and 
ends with
a brief  conclusion.

\section{Problem Formulation and Review of HJE}
\label{sec:2}

\subsection{Discounted deterministic control problem in infinite horizon}
\label{ssec:control}
The   optimal control problem in our study aims at minimizing the cost function with a discount factor $\rho\geq 0$:
\begin{equation}
\label{def:payoff}
J_{x}(u(\cdot)):= \int_{0}^{+\infty}e^{-\rho t}\:l(x(t),u(t)) \dt
\end{equation}
subject to the state equation
\begin{equation}
\label{ODE}
\begin{cases}
\d x(t)=g(x(t),u(t))\dt \\
x(0)=x,
\end{cases}
\end{equation}
where $x(\cdot):\mathbb{R}\rightarrow \mathbb{R}^d$ is the state variable, $u(\cdot):\mathbb{R}\rightarrow \mathbb{R}^p$ is the control function such that $\int_{0}^{+\infty}e^{-\rho t}\| l(x(t),u(t))\| \dt<\infty$ and $u(t)\in\mathcal{U}_{ad}$, a.e. $t$, in which $\mathcal{U}_{ad}$ is an non-empty closed 
convex subset of $\mathbb{R}^{p}$. 

A feedback control $u$ means there is a function $a(\cdot)$ in the state variable $x$:  $\Real^d\to \Real^p$,  such that the control $u(t)=a(x(t))$ with $x(t)$ satisfying  the   ODE \eqref{ODE}
in the autonomous form: $\d x(t) = g(x(t), a(x(t)))\dt$. Throughout  the paper, we shall use $g(x,a)$ and $g(x,u)$ interchangeably for the function $g$. Also $g(\cdot,\cdot):\mathbb{R}^d\times \mathbb{R}^p\rightarrow \mathbb{R}^d$ and $l(\cdot,\cdot):\mathbb{R}^d\times \mathbb{R}^p\rightarrow \mathbb{R}$ have assumptions as below \cite{bensoussan2018estimation}:

\begin{assumption}\label{ass:A1}
There exist some positive constants $\bar{g}$,  $\bar{g}_2$, $\bar{l}$, $\bar{l}_1$, $\bar{l}_2$, $c_0$, $c_s$ and a matrix $c$ in $\Real^{p\times p}$ 
with its norm $\|c\|=\bar{c}$, such that 
	\begin{itemize}
		\item [\textbf{A1.}] $g(x,a)=g_1(x)+c^\top a:\mathbb{R}^{d}\times\mathbb{R}^{p}\to\mathbb{R}^{d}$ and 
		\begin{equation}
		\begin{split}
		&\|g(x,a)\|\leq\bar{g}\left(1+\|x\|+\|a\|\right); \|D_x g(x,a)\|\leq \bar{g}; \\ &\sum_{i=1}^d\|D_{x}(\partial_{x_i} g(x,a))\|\leq \frac{\bar{g}_2}{1+\|x\|};
		\end{split}
		\end{equation}
		where $x_i$ is the $i$-th component of $x$. 

		\item [\textbf{A2.}] $l(x,a): \mathbb{R}^{d}\times\mathbb{R}^{p}\to\mathbb{R}$ is strictly convex and satisfies 
        \begin{equation}
\begin{split}
		&\|l(x,a)\|\leq \bar{l}\left(1+\|x\|^{2}+\|a\|^{2}\right);\\
		&\|l(x,a)-l(x',a')\|\leq \bar{l}\big( (1+ \max(\|x\|,\|x'\|)+ \max(\|a\|,\|a'\|) )\\&\qquad
		 \qquad\qquad\qquad\quad (\|x-x'\|+\|a-a'\|)\big);\\
		&\|\nabla_a l(x,a)\| \geq  \bar{l}_{1}\|a\|-c_{0};\ \left\|(\nabla_a\nabla_a^\top) l(x,a)\right\|\geq c_s;
        \end{split}\end{equation}
and the norm of all the second order derivatives, i.e. $\|(\nabla_x\nabla_a^\top) l(x,a)\|$,\\ $\|(\nabla_a\nabla_x^\top) l(x,a)\|$ and $\|(\nabla_x\nabla_x^\top) l(x,a)\|$ are bounded by $\bar{l}_2$ from above. 
	\end{itemize}
\end{assumption}
The value function  $\Phi(x) $ is defined by 	\begin{equation}
	\label{det_value*}
	\Phi(x) = \inf_{a(\cdot)\in \mathcal{U}_{ad}}J_{x}(a(\cdot)).
	\end{equation}	


{\bf Notations:}
$\nabla$ and $(\nabla\nabla^\top)$ refer to the gradient and Hessian matrix, respectively, of a scalar function. In general, $D$ is used for the derivatives of a
vector-valued function, i.e., the Jacobi matrix. For example, 
$D_x g(x,a)$ refers to 
the   Jacobi matrix in $x$ variable with $(i,j)$ entry $ \frac{\partial g_i}{\partial{x_j}} (x,a) $.
$D^\tr_x g$ means the transpose of the Jacobi matrix $D_x g$. 
	
\subsection{Hamilton-Jacobi equation} \label{ssec:HJE-v-p}
 By the theory of Dynamic Programming, the value function $\Phi(\cdot)$ of \eqref{det_value*} satisfies the 
 (stationary) Hamilton-Jacobi equation (HJE)	
\begin{equation}
\label{lin_HJE}
\rho\Phi(x) = g(x,\hat{a}(x))\cdot \nabla \Phi(x)+l(x,\hat{a}(x)),
\end{equation}
where   the optimal policy is 
\begin{equation}\label{eq:opt-v}
\hat{a}(x)\in \argmin_a~ [ g(x,a)\cdot \nabla \Phi(x)+l(x,a)].
\end{equation}
We drop  out the possible constraint $a\in \mathcal{U}_{ad}$
under $\argmin$ or $\min$ for convenience.
The first equation \eqref{lin_HJE} is 
a {\it linear} stationary  hyperbolic PDE with advection velocity field
$g(x,\hat{a}(x))$. It is the convention to introduce  the Hamiltonian $$ {H}(x, \lambda,a):=g(x,a)\cdot\lambda+l(x,a)$$
and the HJE can be written as
\begin{equation}\label{equ:HJE_compect}\rho \Phi (x)=\min_{a }H(x,\nabla \Phi,a).\end{equation}
     

\subsection{Pontryagin’s maximum  principle (PMP)}
PMP generally refers to 
the first-order necessary
optimality conditions for problems of optimal control \cite{Pontryagin1987}.
For the optimal control problem specified in Section \ref{ssec:control}, the PMP
takes the following form 
\begin{subequations}
\begin{align}
\label{eq:5a}
  \frac{\d }{\dt}  x^{*}(t)&=  H_{\lambda}(x^{*}, \lambda^{*}, u^{*}) =g(x^{*},  u^{*});
\\
\label{eq:5b}
  \frac{\d }{\dt}  (e^{-\rho t} \lambda^{*}(t))&= - e^{-\rho t}  H_{x}(x^{*}, \lambda^{*}, u^{*})
  \notag
  \\ &=- e^{-\rho t}  [ \nabla_{x} l(x^{*}, u^{*}) +   D^{\tr}_{x}g(x^{*},  u^{*})  \lambda^{*}]; \\
\label{eq:5c}
  \frac{\d }{\dt}  (e^{-\rho t}v^{*}(t))&= -e^{-\rho t} l(x^{*}, u^{*}) ; 
\end{align}\end{subequations}
where  $u^{*}(t)\in \mathcal{U}_{ad}$ is defined by $\hat{a}(x^{*}(t))$ in \eqref{eq:opt-v}, i.e., $$\displaystyle u^{*}(t)=\argmin_{u} H(x^{*}(t), u, \lambda^{*}(t)).$$ $\lambda^*(t)$ is the co-state or adjoint variable and $v^*(t)$ is the cost.
Note that 
\eqref{eq:5a} has the initial condition $x^{*}(0)=x$ while 
\eqref{eq:5b} and \eqref{eq:5c}  have the terminal condition vanishing at infinity: 
$ e^{-\rho t} \lambda^*(t)\to 0 $ and $  e^{-\rho t} v^*(t)\to 0 $.

\subsection{Value iteration and policy iteration for HJE}
\label{ssec:PI}
Based on  the   equations
\eqref{lin_HJE} and \eqref{eq:opt-v} as a fixed-point problem for the pair of 
$\Phi$ and $\hat{a}$, many iterative computational methods have been developed in history \cite{Bellman1957,Bellman1957Book,Howzard1960}. 
They can roughly be divided into two categories:
value iteration and policy iteration, which are   central  concepts  
  in reinforcement learning \cite{sutton2018reinforcement}.

   In our model of equation \eqref{equ:HJE_compect}, the value iteration,
   roughly speaking, refers to   the sequence 
   of functions  recursively defined by 
 \begin{equation}\label{eq:value-iter}
\Phi^{(k+1)}(x) := \rho^{-1}\min_{a}
\left [ g(x,a)\cdot \nabla \Phi^{(k)} (x)+l(x,a)\right],~\forall x.\end{equation}
By contrast, the policy iteration  
requires to solve the so-called Generalized HJE. It starts with
an initial policy function $a^{(0)}$ and 
runs the iteration from $a^{(k)}$ to $a^{(k+1)}$ as follows.

\begin{algorithm}
\caption{Policy Iteration (Successive Approximation) for HJE}
\begin{enumerate}
\item Solve the linear PDE \eqref{lin_HJE} for the value function ${\Phi}^{(k+1)}$ with the given policy $\hat{a}=a^{(k)}$: 
\begin{equation} 
\label{779}
\rho {\Phi}^{(k+1)}(x) = g(x,a^{(k)}(x))\cdot \nabla {\Phi}^{(k+1)}(x)+l(x,a^{(k)}(x)).
\end{equation}
This linear equation is referred to as Generalized HJE.\\
\item Then, ${a}^{(k+1)}$
is obtained from the optimization sub-problem \eqref{eq:opt-v} point-wisely for each $x$:
\begin{equation*}
{a}^{(k+1)}(x):=\argmin_a\, [ g(x,a)\cdot \nabla {\Phi}^{(k+1)}(x)+l(x,a)].
\end{equation*}
\end{enumerate}
\end{algorithm}

 Step 1 is usually referred to as   {\it policy evaluation}.
  Step 2 is usually referred to as   {\it policy improvement} 
and $a^{(k+1)}$ is the {\it greedy policy}.
 
 The policy iteration  is known to have super-linear convergence in many   cases provided the initial guess is sufficiently close to the solution and generally behaves better than the value iteration
 \cite{Kalise2015SISC}. The convergence of policy iteration can be found in  \cite{Puterman1979}. 

\section{Formulation  for  Value-Gradient functions}
\label{sec:main}

We start to present our main theoretic results 
and derive the new system of PDEs for the gradient of the value function.

\subsection{Equation for the value-gradient functions}

Define the  {\it  value-gradient} function:
 $$\lambda(x)=\nabla \Phi(x),$$
 then 
the HJE \eqref{lin_HJE} reads
\begin{equation}
\rho\Phi(x) = g(x,\hat{a}(x))\cdot \lambda(x)+l(x,\hat{a}(x)).
\end{equation}
where $x=(x_1,\ldots,x_d)\in\Real^d$.
Now differentiating  both sides w.r.t. $x_i$, we have
\begin{align*}
   \rho \lambda_i (x)
    &=  \sum_n \lambda_n(x)\set{ \left ( \frac{\partial}{\partial{x_i}}+\sum_j \frac{\partial \hat{a}_j}{\partial x_i}
    \frac{\partial }{\partial{a_j}}\right) g_n(x,\hat{a}(x))  }
    \\
    &+ \sum_n g_n(x,\hat{a}(x)) \frac{\partial\lambda_n}{\partial{x_i}} (x)+\left ( \frac{\partial}{\partial{x_i}}+\sum_j \frac{\partial \hat{a}_j}{\partial x_i}
    \frac{\partial }{\partial{a_j}}\right)   l (x,\hat{a}(x)).
\end{align*}
where $\lambda_i$ and $\hat{a}_i$ are the $i$-th component of $\lambda$ and $\hat{a}$ respectively. We assume that the Hamiltonian minimization \eqref{eq:opt-v}
has the unique minimizer $\hat{a}(x)$ which is continuously differential.  
Then the minimizer $\hat{a}(x)$ satisfies the first order necessary condition:
\begin{equation}
  \sum_n \frac{\partial g_n}{\partial {a_j}} (x,\hat{a}) \lambda_n (x)+ \frac{\partial l}{\partial{a_j}}(x,\hat{a})=0, ~~\quad \forall j.
\end{equation}
With the both equalities  above, we have   that $\lambda(x)=(\lambda_1,\ldots,\lambda_d)$  satisfies the following    system of linear hyperbolic
PDEs  
\begin{align}
    \rho \lambda_i 
   =&  \sum_n g_n  \frac{\partial\lambda_n}{\partial{x_i}} 
   +\sum_n   \lambda_n\frac{\partial g_n}{\partial{x_i}}  
+  \frac{\partial l }{\partial{x_i}},
\label{eq:lam-pde-c}
\end{align}
or in the compact form 
\begin{equation}
\begin{split}
\label{eq:lam-pde}
\rho\lambda(x) = & D^\tr \lambda(x) g(x,\hat{a}(x))+D_x^\tr g(x,\hat{a}(x))\lambda(x)+\nabla_xl(x,\hat{a}(x)),
\end{split}
\end{equation}
and  $\hat{a}(x)$ defined  by \eqref{eq:opt-v} can now be written as 
 \begin{equation}\label{eq:opt-v-grad}
\hat{a}(x)= \argmin_a~ [ g(x,a)\cdot\lambda (x)+l(x,a)].
\end{equation}
\eqref{eq:lam-pde} and \eqref{eq:opt-v-grad} 
are coupled as \eqref{lin_HJE} and \eqref{eq:opt-v} 
in the HJE and they serve as the foundation   for the new development of the algorithms,
based on the policy iteration method.
 
 Given a policy $\hat{a}$, the system of coupled PDEs \eqref{eq:lam-pde} is a closed form involving only the dynamic function $g$ and the 
 running cost function $l$; it does not need other information like the value function.
 It plays the similar role to  the Generalized HJE  \eqref{lin_HJE} for the value function $\Phi$.
  \eqref {eq:lam-pde} and \eqref{eq:opt-v-grad} together can replace the 
 traditional dynamic programming in the form of  HJE
 if $\Phi$ is sufficiently smooth.
 The main focus of our work is 
 how to develop efficient numerical methods from this formulation of 
 the gradient of the value function.
 
 Since $\lambda(x)$ is  the gradient of the value function $\Phi$,
 so $D\lambda(x) = \nabla^2 \Phi(x)$  should be symmetric, i.e., $D\lambda =D^\tr \lambda$. 
Then the value-gradient satisfies
\begin{equation}
\label{eq:lambda-c}
\begin{split}
    \rho \lambda_i (x)
&= \nabla  \lambda_i(x) \cdot g (x, \hat{a}(x))
   +\sum_n   \frac{\partial g_n}{\partial{x_i}}   \lambda_n(x)
+  \frac{\partial l }{\partial{x_i}} (x,\hat{a}(x))
\end{split}
\end{equation}
or 
\begin{equation}
\label{eq:lambda}
\begin{split}
\rho\lambda(x) &=  (D  \lambda) g + (D_x^\tr  g) \lambda(x)+ \nabla_xl .
\end{split}
\end{equation}
where $\hat{a}(x)$ is defined in \eqref{eq:opt-v-grad}
as the unique minimizer of the Hamiltonian $H(x,\lambda(x))$.
In addition,  if $\lambda(x)$ satisfies the systems of PDEs \eqref{eq:lambda-c}, then 
for $x^*$ as the optimal trajectory satisfying  the characteristics equation \eqref{eq:5a},
then $\lambda^*(t):=\lambda( x^*(t))$ satisfies the   equation \eqref{eq:5b}. The conclusion that $\lambda^*(t):=\lambda( x^*(t))$ satisfies  \eqref{eq:5b}
follows from the following fact   
$$\frac{\d}{\dt} \lambda(t) = (D\lambda) g  = \rho \lambda^*(t) -
\left[(D_x^\tr  g) \lambda(x^*)+ \nabla_xl\right].$$

The advantage of     equation \eqref{eq:lambda-c} over the equation \eqref{eq:lam-pde-c} 
is that the advection terms $D\lambda_i \cdot g $  are now decoupled for each 
component $i$ and   the same as in the GHJE \eqref{lin_HJE}. This property will 
allow us to develop a fully paralleled iterative method.

\subsection{Policy iteration  for value gradient}
The natural  idea to solve the PDEs for \eqref{eq:lambda-c} and 
the minimization for $\hat{a}$ in \eqref{eq:opt-v-grad}
is the policy iteration by recursively solving \eqref{eq:lambda-c}
and \eqref{eq:opt-v-grad} like the policy iteration for the value function
dictated in Section \ref{ssec:PI}:
 Start with
an initial policy function $a^{(0)}$ with $k=0$;
\begin{enumerate}
    \item Solve the system \eqref{eq:lambda-c} with the given policy $\hat{a}=a^{(k)}$  to have 
     ${\lambda}^{(k+1)}$;
\item   ${a}^{(k+1)}$
is obtained from the optimization sub-problem \eqref{eq:opt-v}.
\end{enumerate}
This iteration will produce a sequence of pairs 
$(a^{(k)},\lambda^{(k)})$, $k\geq 1 $. 
 The main task is then to solve \eqref{eq:lambda-c} (or \eqref{eq:lambda}), the system of linear PDEs for $\lambda(x)$,
with a given policy $a$. We will first propose  the method for this  system of linear PDEs and  
more details are given in Section \ref{sec:nm}. We summarize our main algorithm {\bf poicy iteration based on $\lambda$ (PI-lambda)} as below.

\begin{algorithm}
\caption{{\bf   PI-lambda}: policy iteration based on $\lambda$}
\begin{enumerate}
\item For $i=1,\ldots, d$, solve the PDE for each $\lambda^{(k+1)}_i$ in parallel
\begin{align}
    \rho \lambda^{(k+1)}_i(x)-&D\lambda^{(k+1)}_i (x) \cdot g(x, {a}^{(k)}(x))\notag\\
    &=\sum_n   \frac{\partial g_n}{\partial{x_i}}  \lambda^{(k)}_n(x)+   \frac{\partial l }{\partial{x_i}}(x, {a}^{(k)}(x)), \label{eq:iterlambda2}
\end{align}
 with the given policy $\hat{a}=a^{(k)}$  to have 
     ${\lambda}^{(k+1)}=({\lambda}^{(k+1)}_1,\ldots,{\lambda}^{(k+1)}_d)$;
\item ${a}^{(k+1)}$
is obtained from the optimization sub-problem \eqref{eq:opt-v}:
\begin{equation*}
{a}^{(k+1)}(x)=\argmin_a\, [ g(x,a)\cdot \lambda^{(k+1)}(x)+l(x,a)].
\end{equation*}
\end{enumerate}
\end{algorithm}


\medskip

The   merit 
 of  \eqref{eq:iterlambda2} is that the components of $\lambda^{(k+1)}(x)$ are completely decoupled and can be solved in parallel.  Each equation of these $d$ components is exactly in the same 
form as the GHJE \eqref{779} for the value function.
So the method of characteristics, which will be detailed 
in the next section,  can be applied to both the GHJE \eqref{779} and 
the system \eqref{eq:iterlambda2}.




\subsection{Convergence analysis for PI-lambda}
In this subsection, we will state and prove our main theorem \cref{thm: main} that  PI-lambda algorithm converges linearly in  $L_\alpha^2$ sense (see \eqref{equ:sense}) for a suitable choice of exponent $\alpha$ in a weight factor. The proof will need two important lemmas, \cref{lem:iter} and \cref{lem:bound}, which are proved in the supplymentary materials. $\lambda^{(k)}(x)$ and $a^{(k)}(x)$ stand for the value-gradient and control function of the $k$-th iteration in  PI-lambda respectively.

\begin{lemma}\label{lem:iter}
Under Assumptions \ref{ass:A1}, at the $k$-th iteration of value-gradient, if there exist constants $\bar{\lambda}^{(k)}$, $\bar{\lambda}^{\prime(k)}$,  $\bar{a}^{(k)}$, $\bar{a}^{\prime(k)}$ such that
\begin{align*}
    \|\lambda^{(k)}(x)\|  &\leq \bar{\lambda}^{(k)}(1+\|x\|), \left\|D\lambda^{(k)}(x)\right\| \leq \bar{\lambda}^{\prime(k)},\\
    \|a^{(k)}(x)\|  &\leq \bar{a}^{(k)} (1+\|x\|),\ \left\|Da^{(k)}(x)\right\| \leq \bar{a}^{\prime(k)},
\end{align*}
and if
$$\rho>\bar{g}(1+\bar{a}^{(k)})+\bar{c}\bar{a}^{\prime (k)},$$then
\begin{align*}
    \|\lambda^{(k+1)}(x)\|  &\leq \bar{\lambda}^{(k+1)}(1+\|x\|), \left\|D\lambda^{(k+1)}(x)\right\| \leq \bar{\lambda}^{\prime(k+1)},\\
     \|a^{(k+1)}(x)\|  &\leq \bar{a}^{(k+1)} (1+\|x\|),\ \left\|Da^{(k+1)}(x)\right\| \leq \bar{a}^{\prime(k+1)},
\end{align*}
where the constants
\begin{align}
    \bar{\lambda}^{(k+1)}&=\frac{\bar{l}+\bar{l} \bar{a}^{(k)}+\bar{g}\bar{\lambda}^{(k)}}{\rho-\bar{g}(1+\bar{a}^{(k)})}>0,\label{equ:barlam}\\
        \bar{\lambda}^{\prime(k+1)}&=\frac{\bar{l}_{2}+\bar{l}_{2}\bar{a}'^{(k)}+\bar{g}_2\bar{\lambda}^{(k)}+\bar{g}\bar{\lambda}'^{(k)}}{\rho-(\bar{g}+\bar{c}\bar{a}'^{(k)})}>0,\label{equ:barlamprime}\\
    \bar{a}^{(k+1)}&=\frac{\bar{c} \bar{\lambda}^{(k+1)}+c_0}{\bar{l}_1},\label{equ:bara}\\
    \bar{a}'^{(k+1)}&=\frac{\bar{l}_{2}+\bar{c}\bar{\lambda}^{\prime(k+1)}}{c_s}.\label{equ:baraprime}
\end{align}
\end{lemma}

\begin{proof}At the $k$-th iteration, suppose that 
\begin{align*}
    \|\lambda^{(k)}(x)\|&\leq \bar{\lambda}^{(k)}(1+\|x\|),\ \left\|D\lambda^{(k)}(x)\right\|\leq \bar{\lambda}'^{(k)}\\
    \|a^{(k)}(x)\|  &\leq \bar{a}^{(k)} (1+\|x\|),\ \left\|Da^{(k)}(x)\right\| \leq \bar{a}'^{(k)},
\end{align*}
 where $\bar{\lambda}^{(k)},\bar{\lambda}'^{(k)},\bar{a}^{(k)},\bar{a}'^{(k)}$ are all constants.\\
 We first bound $\|X^{(k)}(t)\|$. Here $X^{(k)}(t)$ is a trajectory with dynamic system
 \begin{equation*}
 \begin{split}
 \dot{X}^{(k)}(t)&=g(X^{(k)}(t),a^{(k)}(X^{(k)}(t))),\\ X^{(k)}(0)&=x\in\mathbb{R}^p,    
 \end{split}
\end{equation*}
And we have
\begin{equation}\label{equ:gron_1}
\begin{split}
\|X^{(k)}(t)\|+1=&\left\|x+\int_0^t \dot{X}^{(k)}(s)\d s\right\|+1 \leq \|x\|+1+\int^t_0\left\|\dot{X}^{(k)}(s)\right\|\d s\\
\leq &\|x\|+1+\int^t_0\|g(X^{(k)}(s))\|\d s\\
\leq &\|x\|+1+\int^t_0\bar{g} \left(1+\|X^{(k)}(s)\|+\|a^{(k)}(X^{(k)}(s))\|\right)\d s\\
\leq &\|x\|+1+\int^t_0\bar{g} \left(1+\|X^{(k)}(s)\|+\bar{a}^{(k)}(1+\|X^{(k)}(s)\|)\right)\d s\\
\leq &\|x\|+1+\int^t_0\bar{g} (1+\bar{a}^{(k)})(\|X^{(k)}(s)\|+1)\d s
\end{split}
\end{equation}
By the Grönwall's inequality
\begin{equation}
\eqref{equ:gron_1} \leq e^{\bar{g}(1+\bar{a}^{(k)})t}(\|x\|+1).
\end{equation}
For any $k=0,1,2,...$

\begin{equation*}
\begin{split}
    \d\left[e^{-{\rho t}}\lambda^{(k+1)}\left(X^{(k)}(t)\right)\right]=&-e^{-\rho t}\bigg[\nabla_x l\left(X^{(k)}(t),a^{(k)}\left(X^{(k)}(t)\right)\right)\\&+D_x g\left(X^{(k)},a^{(k)}\left(X^{(k)}(t)\right)\right)\lambda^{(k)}(X^{(k)}(t))\bigg] \d t.
\end{split}
\end{equation*}
For any initial $x$, integrate on both sides from $0$ to $\infty$ w.r.t $t$, we have
\begin{equation}
\begin{split}
\label{equ:lam_tra}
    \lambda^{(k+1)}(x)=&\lim_{t\to \infty} e^{-\rho t}\lambda^{(k+1)}(X^{(k)}(t))+\int^{+\infty}_0 e^{-\rho s}\bigg[\nabla_x l(X^{(k)}(s),\\&a^{(k)}(X^{(k)}(s)))
    +D_x g(X^{(k)}(s),a^{(k)}(X^{(k)})(s))\lambda^{(k)}(X^{(k)}(s))\bigg]\d s
    \end{split}
\end{equation}
We consider the physical solution, and $\lambda^{(k+1)}$ is at most polynomial growth. Here we apply method of undetermined coefficients.  Suppose $\lambda^{(k+1)}(x)$ satisfies $$ \|\lambda^{(k+1)}(X^{(k)}(t))\|\leq M_{k}(1+\|X^{(k)}(t)\|^{m_k})$$ 
when $t\to \infty$, where $M_{k}\geq 0$ and $m_k\geq 0$ are constants to be determined. Take norm on both sides of $\eqref{equ:lam_tra}$ yield
\begin{equation*}
\begin{split}
   \|\lambda^{(k+1)}(x)\| \leq& \lim_{t\to \infty} e^{-\rho t}M_{k}(1+\|X^{(k)}(t)\|^{m_k})+\int^{+\infty}_0\bigg\| e^{-\rho s}\bigg[\nabla_x l(X^{(k)},a^{(k)}(X^{(k)}(s)))\\&+D_x g(X^{(k)}(s),a^{(k)}(X^{(k)})(s))\lambda^{(k)}(X^{(k)}(s))\bigg]\bigg\|\d s.
 \end{split}
\end{equation*} Select $\rho$ such that $\rho> m_k(\bar{g}(1+\bar{a}^{(k)}))$, then we have
$$\lim_{t\to \infty} e^{-\rho t}M_{k}(1+\|X^{(k)}(t)\|^{m_k})=0.$$ According to the assumption, there holds
\begin{equation}
    \left\|\nabla_x l(x,a)\right\|=\bar{l}(1+\|x\|+\|a\|).
\end{equation}
Thus  
\begin{equation}
\begin{split}
\|\lambda^{(k+1)}&(x)\|
 \leq  \int^{+\infty}_0 e^{-\rho s}\big[\bar{l}(1+\|X^{(k)}(s)\|+\|u^{(k)}(X^{(k)}(s))\|)+\bar{g}\|\lambda^{(k)}(X^{(k)}(s))\|\big]\d s\\
     \leq &\int^{\infty}_0 e^{-\rho s}\bigg[\bar{l}+\bar{l}\|X^{(k)}(s)\|+\bar{l} \bar{a}^{(k)}(1+\|X^{(k)}(s)\|+\bar{g}\bar{\lambda}^{(k)}(1+\|X^{(k)}(s)\|))\bigg]\d s\\
    \leq &\int^{+\infty}_0 e^{-\rho s}(\bar{l}+\bar{l}\bar{a}^{(k)}+\bar{g}\bar{\lambda}^{(k)})(1+\|X^{(k)}(s)\|)\d s\\
    \leq &\int^{+\infty}_0 e^{-\rho s}(\bar{l}+\bar{l} \bar{a}^{(k)}+\bar{g}\bar{\lambda}^{(k)}) e^{\bar{g}(1+\bar{a}^{(k)})s}(1+\|x\|)\d s\\
    =&\frac{\bar{l}+\bar{l} \bar{a}^{(k)}+\bar{g}\bar{\lambda^{(k)}}}{\rho-\bar{g}(1+\bar{a}^{(k)})}(1+\|x\|)\\
    =: & \bar{\lambda}^{(k+1)}(1+\|x\|).
\end{split}
\end{equation}
where $\rho>\bar{g}(1+\bar{a}^{(k)})$. So we only need $m_k=1$, and $M_{k}$ be any real number larger than
$\bar{\lambda}^{(k+1)}$. Thus proves \eqref{equ:barlam}.
At each iteration, $a^{(k)}$ is solved by
\begin{equation}
    \nabla_a l\left(x,a^{(k)}(x)\right)+c^\top \lambda^{(k)}(x)=0\label{equ:solveu}
\end{equation} 
We have
\begin{equation*}
    \bar{c}\|\lambda ^{(k+1)}(x)\|\geq \left\|\nabla_a l(x,a^{(k+1)}(x))\right\| \geq \bar{l}_1 \|a^{(k+1)}(x)\|-c_0.
\end{equation*}
And
\begin{equation*}
    \|a^{(k+1)}(x)\|\leq \frac{(\bar{c} \bar{\lambda}^{(k+1)}+c_0)(1+\|x\|)}{\bar{l}_1} =:\bar{a}^{(k+1)}(1+\|x\|),
\end{equation*}  which proves \eqref{equ:bara}. Next, we consider $\left\|D\lambda^{(k+1)}(x)\right\|$ and $\left\|D a^{(k+1)}(x)\right\|$. We begin with bounding $\left\|D_x X^{(k)}(t)\right\|$.
\begin{equation*}
\begin{split}
    \frac{\d (D_x X^{(k)}(t))}{\d t}=&\bigg[D_x g(X^{(k)}(t),a^{(k)}(X^{(k)}(t))) \\
    &\qquad+D_a g(X^{(k)}(t),a^{(k)}(X^{(k)}(t)))Da^{(k)}(X^{(k)}(t))\bigg]D_x X^{(k)}(t).
    \end{split}
\end{equation*}
And there holds 
\begin{equation}\begin{split}\label{equ:gron_2}
\left\|D_x X^{(k)}(t)\right\|=&\left\|D_x X^{(k)}(0)+\int_0^t\frac{\d (D_x X^{(k)}(s))}{\d s}\d s\right\|\\
\leq &\|D_x X^{(k)}(0)\|+\int^t_0\left\|\frac{\d (D_x X^{(k)}(s))}{\d s}\right\|\d s\\
\leq &1+\int^t_0 (\bar{g}+\bar{c}\bar{a}'^{(k)})\|D_x X^{(k)}(s)\|\d s
\end{split}
\end{equation}
By the Grönwall's inequality, 
\begin{equation}
\begin{split}
\eqref{equ:gron_2} \leq e^{(\bar{g}+\bar{c}\bar{a}'^{(k)})t }.
\end{split}
\end{equation}
Take the derivative of \eqref{equ:lam_tra}
\begin{equation}
\begin{split}
\label{equ:Dlam_tra}
D\lambda^{(k+1)}(x)&
  =\lim_{t\to \infty}e^{-\rho t}D\lambda^{(k+1)}(X^{(k)}(t)) D_x X^{(k)}(t)\\
 & +\int^{+\infty}_0 e^{-\rho s}\bigg[\nabla^2_{xx} l\left(X^{(k)}(s),a^{(k)}\left(X^{(k)}(s)\right)\right)D_x X^{(k)}(s)+\\ & \nabla^2_{xa} l\left(X^{(k)}(s),a^{(k)}\left(X^{(k)}(s)\right)\right)Da^{(k)}\left(X^{(k)}(s)\right) D_x X^{(k)}(s)\\
 &+\sum^d_{i=1}\lambda_i(X^{(k)}(s))D_x\left(\partial_{x_i}g\left(X^{(k)}(s),a^{(k)}(X^{(k)}(s))\right)\right)D_x X^{(k)}(s)\\ & +D_x g\left(X^{(k)}(s),a^{(k)}\left(X^{(k)}(s)\right)\right)D\lambda^{(k)}\left(X^{(k)}(s)\right)D_x X^{(k)}(s)\bigg]\d s
 \end{split} \end{equation}Likewise, we use method of determined coefficients here. Suppose $D\lambda^{(k+1)}(X^{(k)}(t))$ satisfies $$\left\|D\lambda^{(k+1)}(X^{(k)}(t))D_x X^{(k)}(t)\right\|\leq   N_k(1+\|X^{(k)}(t)\|^{n_k})$$ 
when $t\to \infty$, and $N_{k}\geq 0$ and $n_k\geq 0$ are constants to be determined. Take norm on both sides of \eqref{equ:Dlam_tra} gives
  \begin{equation}
\begin{split}
 \left\|D\lambda^{(k+1)}(x)\right\| \leq &\lim_{t\to \infty}e^{-\rho t}\left\|D\lambda^{(k+1)}(X^{(k)}(t))D_x X^{(k)}(t)\right\|\\&+\bigg\|\int^{+\infty}_0 e^{-\rho s} \bigg[\nabla^2_{xx} l\left(X^{(k)}(s), a^{(k)}\left(X^{(k)}(s)\right)\right) D_x X^{(k)}(s)\\&+\nabla^2_{xa} l\left(X^{(k)}(s),a^{(k)}\left(X^{(k)}(s)\right)\right)  Da^{(k)}\left(X^{(k)}(s)\right) D_x X^{(k)}(s)\\&+\sum^d_{i=1}\lambda_i\left(X^{(k)}(s)\right)D_x\left(\partial_{x_i}g\left(X^{(k)}(s),a^{(k)}(X^{(k)}(s))\right)\right)D_x X^{(k)}(s)\\& +D_x g\left(x^{(k)}(s),a^{(k)}\left(X^{(k)}(s)\right)\right)D\lambda^{(k)}(X^{(k)}(s))D_x X^{(k)}(s)\bigg]\d s\bigg\|.
\end{split}
\end{equation}
Select $\rho$ such that $\rho> n_k(\bar{g}(1+\bar{a}^{(k)})+\bar{c}\bar{a}'^{(k)})$, then we have
\begin{equation*}\lim_{t\to \infty}e^{-\rho t}\Big\|D\lambda^{(k+1)}(X^{(k)}(t))D_x X^{(k)}(t)\Big\|=0,\end{equation*}
and
\begin{equation}
\begin{split}
  \left\|D\lambda^{(k+1)}(x)\right\| 
  \leq &\int_0^{+\infty} e^{-\rho s}\bigg\{\bar{l}_{2}\|D_x X^{(k)}(s)\|+\frac{\bar{g}_2}{1+\|x\|}\bar{\lambda}^{(k)}(1+\|x\|)\\&+\bar{l}_{2}\left\|Da^{(k)}(X^{(k)}(s))\right\|\left\|D_x X^{(k)}(s)\right\|
  \\&+\bar{g}\left\|D_x X^{(k)}(s)\right\|\left\|D\lambda^{(k)}(X^{(k)}(s))\right\|\left\|D_x X^{(k)}(s)\right\|\bigg\}\d s\\
    \leq & \int^{+\infty}_0 e^{-\rho s} e^{(\bar{g}+\bar{c}\bar{a}'^{(k)})s}\Big[\bar{l}_{2}+\bar{l}_{2}\bar{a}'^{(k)}+\bar{g}_2\bar{\lambda}^{(k)}+\bar{g}\bar{\lambda}'^{(k)}\Big]\d s\\
    \leq & \frac{\bar{l}_{2}+\bar{l}_{2}\bar{a}'^{(k)}+\bar{g}_2\bar{\lambda}^{(k)}+\bar{g}\bar{\lambda}'^{(k)}}{\rho-(\bar{g}+\bar{c}\bar{a}'^{(k)})}\\ =: &\bar{\lambda}'^{(k+1)}
    \end{split}
\end{equation}
where $\rho>\bar{g}+\bar{c}\bar{a}'^{(k)}$. So we only need $n_k=0$ and $\bar{N}_{k}$ be any constant larger than $\bar{\lambda}'^{(k+1)}$. This proves \eqref{equ:barlamprime}.\\
Recall that $a^{(k)}$ is solved by \eqref{equ:solveu}. Due to the strict convexity of $l$, the control has unique solution. Take the derivative w.r.t. $x$, 
\begin{equation*}
   (\nabla_a\nabla_x^\top) l(x,a^{(k)}(x))+(\nabla_a\nabla_a^\top) l(x,a^{(k)}(x))Da^{(k)} (x)+c^\top D\lambda^{(k)}(x)=0.
\end{equation*} 
Consider $ \|(\nabla_a\nabla_x^\top) l (x,a^{(k)}(x))\|\leq \bar{l}_{2}$ and $\|(\nabla_a\nabla_a^\top)l(x,a^{(k)}(x))\|>c_s$, we have 
\begin{equation*}\left\|Da^{(k)}(x)\right\|\leq \frac{\bar{l}_{2}+\bar{c}\left\|D\lambda^{(k)}(x)\right\|}{c_s}\leq\frac{\bar{l}_{2}+\bar{c}\bar{\lambda}'^{(k+1)}}{c_s}=: \bar{a}'^{(k+1)},\end{equation*} which proves \eqref{equ:baraprime}.
\end{proof}

\begin{lemma}
\label{lem:bound}
There exist a constant $\rho_1$ such that the sequence $\{\bar{a}^{(k)}\}$, $\{\bar{a}^{\prime(k)}\}$, $\{\bar{\lambda}^{(k)}\}$, $\{\bar{\lambda}^{\prime(k)}\}$ in Lemma \ref{lem:iter} are uniformly bounded by constants $C_1$, $C_2$, $C_3$, $C_4$ respectively if $\rho>\rho_1$ and the initial satisfies
 \begin{equation*}
 \begin{split}
    \|\lambda^{(0)}(x)\|  &\leq C_3(1+\|x\|),\ \left\|D\lambda^{(0)}(x)\right\| \leq C_4,
    \end{split}
 \end{equation*}where the constants are
\begin{align*}
    C_1&=\sqrt{\frac{\bar{c}\bar{l}(1+\frac{c_0}{\bar{l}_1})}{\bar{g}\bar{l}_1}}+\frac{c_0}{\bar{l}_1};\ \   C_2=\frac{1}{c_s}\left(\bar{l}_{2}+\sqrt{c_s\bar{l}_{2}+\bar{l}_{2}^2+\bar{g}_2\sqrt{\frac{c_s(\bar{l}^2+c_0\bar{l})}{\bar{g}}}}\right);\\
     C_3&=\sqrt{\frac{\bar{l}_1\bar{l}(1+\frac{c_0}{\bar{l}_1})}{\bar{g}\bar{c}}};\ \qquad \ C_4=\frac{1}{\bar{c}}\sqrt{c_s\bar{l}_{2}+\bar{l}_{2}^2+\bar{g}_2\sqrt{\frac{c_s(\bar{l}^2+c_0\bar{l})}{\bar{g}}}}.
\end{align*}
\end{lemma}
\begin{proof}In the proof below, we'll show the four sequences $\{\bar{\lambda}^{(k+1)}\}$, $\{\bar{\lambda}'^{(k)}\}$, $\{\bar{a}^{(k)}\}$, $\{\bar{a}'^{(k)}\}$ are uniformly bounded respectively.\\
Take \begin{equation*}\begin{split}\rho_1=\bar{g}(1+C_1)+\bar{c}C_2+2\bar{g}+\frac{\bar{g}c_0}{\bar{l}_1}+\frac{\bar{l}\bar{c}}{\bar{l}_1}+2\sqrt{\bar{l}\left(1+\frac{c_0}{\bar{l}_1}\right)\frac{\bar{g}\bar{c}}{\bar{l}_1}}+\frac{2\bar{c}\bar{l}_2}{c_s}\\
+2\sqrt{\left(\bar{l}_{2}+\frac{\bar{l}_{2}^2}{c_s}+\bar{g}_2\sqrt{\frac{\bar{l}^2+c_0 \bar{l}}{\bar{g}c_s}}\right)\frac{\bar{c}^2}{c_s}}.\end{split}
\end{equation*}Let $\rho>\rho_1$.\\
(a) First, we prove that $\{\bar{\lambda}^{(k)}\}$ is bounded. Bring \eqref{equ:bara} to \eqref{equ:barlam},
$$\bar{\lambda}^{(k+1)}=\frac{\bar{l}+\frac{\bar{l}\bar{c}}{\bar{l}_1}\bar{\lambda}^{(k)}+\frac{\bar{l}c_0}{\bar{l}_1}+\bar{g}\bar{\lambda}^{(k)}}{\rho-\bar{g}(1+\frac{\bar{c} \bar{\lambda}^{(k)}+c_0}{\bar{l}_1})}=\frac{\bar{l}+\frac{\bar{l}c_0}{\bar{l}_1}+(\frac{\bar{l}\bar{c}}{\bar{l}_1}+\bar{g})\bar{\lambda}^{(k)}}{\rho-\bar{g}(1+\frac{c_0}{\bar{l}_1})-\frac{\bar{g}\bar{c}}{\bar{l}_1}\bar{\lambda}^{(k)}}$$
Let
\begin{equation*}
    \begin{split}
    &A=\bar{l}+\frac{\bar{l}c_0}{\bar{l}_1},\ \ \ \ \ \ \ \ \ \ \ \ \ \ \ \ \ \  B=\frac{\bar{l}\bar{c}}{\bar{l}_1}+\bar{g}\\
    &C_\rho=\rho-\bar{g}(1+\frac{c_0}{\bar{l}_1})>0,\ \ D=\frac{\bar{g}\bar{c}}{\bar{l}_1}.\ \
\end{split}
\end{equation*}
Then
$$\bar{\lambda}^{(k+1)}=\frac{A+B\bar{\lambda}^{(k)}}{C_\rho-D\bar{\lambda}^{(k)}}=-\frac{B}{D}+\frac{A+\frac{C_\rho B}{D}}{C_\rho-D\bar{\lambda}^{(k)}}\coloneqq h(\bar{\lambda}^{(k)}).$$ 
To guarantee this iteration form has positive fix point $\bar{\lambda}>0$, solution for $\bar{\lambda}=h(\bar{\lambda})$ should be positive. Since $\rho$ satisfies
$$C_\rho>2\sqrt{AD}+B,$$
then the solutions 
\begin{equation*}
\begin{split}
    \bar{\lambda}_{1,2}&=\frac{C_\rho-B\pm \sqrt{C_\rho^2-2C_\rho B+B^2-4AD}}{2D}
    \end{split}
\end{equation*} 
are positive. Next we prove that sequence $\{\bar{\lambda}^{ (k)}\}$ is bounded in the region $\left(0,\frac{\sqrt{AD}}{D}\right)$. Notice that $0<h'(\bar{\lambda})<1$ on the region. Since $h'(\bar{\lambda})>0$, the function is monotonically increasing on the region. Thus
$$h(0)<\bar{\lambda}^{(k+1)}<h\Big(\frac{\sqrt{AD}}{D}\Big).$$
The lower bound satisfies
\begin{equation*}
    h(0)=-\frac{B}{D}+\frac{A+\frac{C_\rho B}{D}}{C_\rho}=\frac{A}{C_\rho}>0.
\end{equation*}The upper bound satisfies
$$
\begin{aligned}
h\Big(\frac{\sqrt{AD}}{D}\Big)&=-\frac{B}{D}+\frac{A+\frac{BC_\rho}{D}}{C_\rho-\sqrt{AD}}=\frac{B\sqrt{AD}+AD}{D(C_\rho-\sqrt{AD})}\\
&<\frac{\sqrt{AD}(B+\sqrt{AD})}{D(B+\sqrt{AD})}=\frac{\sqrt{AD}}{D}\end{aligned}$$
As a result, $\bar{\lambda}^{(k+1)}\in (0,\frac{\sqrt{AD}}{D})$. The sequence $\{\bar{\lambda}^{(k)}\}$ is bounded by \begin{equation*}
\begin{split}
C_3&\coloneqq\frac{\sqrt{AD}}{D}=\sqrt{\frac{\bar{l}_1\bar{l}(1+\frac{c_0}{\bar{l}_1})}{\bar{g}\bar{c}}}\end{split}\end{equation*} for all $k\in \mathbb{Z}^{+}$.\\
(b) Then, we prove that $\{\bar{\lambda}'^{(k)}\}$ is bounded. Bring \eqref{equ:baraprime} to \eqref{equ:barlamprime},
\begin{align*} \bar{\lambda}'^{(k+1)}&=\frac{\bar{l}_{2}+\bar{l}_{2}\bar{a}'^{(k)}+\bar{g}_2 \bar{\lambda}^{(k)}+\bar{g}\bar{\lambda}'^{(k)}}{\rho-(\bar{g}+\bar{c}\bar{a}'^{(k)})}\\&=\frac{\left(\bar{l}_{2}+\bar{l}_{2}\frac{\bar{l}_{2}}{c_s}+\bar{g}_2\sqrt{\frac{\bar{l}^2+c_0 \bar{l}}{\bar{g}c_s}}\right)+\left(\bar{l}_{2}\frac{\bar{c}}{c_s}+\bar{g}\right)\bar{\lambda}'^{(k)}}{\left(\rho-\bar{g}-\bar{c}\frac{\bar{l}_{2}}{c_s}\right)-\frac{\bar{c}^2}{c_s}\bar{\lambda}'^{(k)}}.\end{align*}
Let
\begin{align*}
    A'&=\bar{l}_{2}+\frac{\bar{l}_{2}^2}{c_s}+\bar{g}_2\sqrt{\frac{\bar{l}^2+c_0 \bar{l}}{\bar{g}c_s}},\ &B'=\frac{\bar{l}_{2}\bar{c}}{c_s}+\bar{g};\\
    C'_\rho&=\rho-\bar{g}-\frac{\bar{l}_{2}\bar{c}}{c_s}, &D'=\frac{\bar{c}^2}{c_s}.\qquad 
\end{align*}Thus
$$\bar{\lambda}^{\prime(k+1)}=\frac{A^\prime+B^\prime\bar{\lambda}^{^\prime(k)}}{C_\rho^\prime-D^\prime\bar{\lambda}^{^\prime(k)}}=-\frac{B^\prime}{D^\prime}+\frac{A^\prime+\frac{C_\rho^\prime B^\prime}{D^\prime}}{C_\rho^\prime-D^\prime\bar{\lambda}^{^\prime(k)}}.$$ It can be show that $\bar{\lambda}^{\prime (k+1)}\in \left(0,\frac{\sqrt{A^\prime D^\prime}}{D^\prime}\right)$ using the same method as in (a). It is easy to obtain that the sequence $\{\bar{\lambda}^{(k)}\}$ is bounded by 
\begin{equation*}
\begin{split}
C_4&\coloneqq\frac{1}{\bar{c}}\sqrt{c_s\bar{l}_{2}+\bar{l}_{2}^2+\bar{g}_2\sqrt{\frac{c_s\left(\bar{l}^2+c_0\bar{l}\right)}{\bar{g}}}}.
\end{split}
\end{equation*}
for all $k\in \mathbb{Z}^{+}$.\\
(c) By \eqref{equ:bara}, we have 
\begin{equation*}
    \bar{a}^{(k)}=\frac{\bar{c} \bar{\lambda}^{(k)}+c_0}{\bar{l}_1}\leq \frac{\bar{c} C_3+c_0}{\bar{l}_1},
\end{equation*} and 
the sequence $\{\bar{a}^{(k)}\}$ is bounded by 
\begin{equation*}
\begin{split}
C_1&\coloneqq\frac{\bar{c} C_3+c_0}{\bar{l}_1}=\sqrt{\frac{\bar{c}\bar{l}\left(1+\frac{c_0}{\bar{l}_1}\right)}{\bar{g}\bar{l}_1}}+\frac{c_0}{\bar{l}_1}\end{split}\end{equation*}for all $k\in \mathbb{Z}^{+}$.\\
(d) By \eqref{equ:baraprime}, we have 
\begin{equation*}
\begin{split}
        \bar{a}'^{(k)}&=\frac{\bar{l}_{2}+\bar{c}\bar{\lambda}^{\prime(k)}}{c_s}\leq \frac{\bar{l}_{2}+\bar{c}C_4}{c_s}, \end{split}
\end{equation*} and 
the sequence $\{\bar{a}^{\prime(k)}\}$ is bounded by 
\begin{equation*}
\begin{split}
C_2&\coloneqq\frac{\bar{l}_{2}+\bar{c}C_4}{c_s}\\
&=\frac{1}{c_s}\left(\bar{l}_{2}+\sqrt{c_s\bar{l}_{2}+\bar{l}_{2}^2+\bar{g}_2\sqrt{\frac{c_s(\bar{l}^2+c_0\bar{l})}{\bar{g}}}}\right)\end{split}\end{equation*}
for all $k\in \mathbb{Z}^{+}$.
\end{proof}

Next we state our main theorem that shows the convergence of {\bf PI-lambda} algorithm.
\begin{theorem}\label{thm: main}
Under Assumption \ref{ass:A1}, for any $\alpha>1$, there exists a large enough $\rho_2$, such that if $\rho>\rho_2$, define
\begin{equation}\label{equ:sense}
      e^{(k)}:=\int_{\mathbb{R}^d}  \frac{\|\lambda^{(k)}(x)-\lambda^{(k-1)}(x) \|^2}{(1+\|x\|^2)^{2\alpha}}\dx.
\end{equation}we have $e^{(k+1)}\leq\eta e^{(k)}$ with $\eta\in(0,1)$. Therefore $\{\lambda^{(k)}\}$ forms a Cauchy-sequence in $L_\alpha^2$-sense.
\end{theorem}
\textbf{Remark}:
Note that $\rho_2\geq\rho_1$, which suggests that if $\rho$ satisfies the inequality condition in \cref{thm: main}, then it satisfies the inequality condition in \cref{lem:bound}.
\begin{proof} 
Recall in equation \eqref{eq:iterlambda2}, $\lambda^{(k)}(x)$ and $\lambda^{(k+1)}(x)$ are defined by
\begin{equation}
\label{equ: grad_pde}
\begin{split}
     \rho \lambda^{(k)}(x)=&D\lambda^{(k)}(x)g\left(x,a^{(k-1)}(x)\right)\\
     &+D_x g\left(x,a^{(k-1)}(x)\right)\lambda^{(k-1)}(x)+\nabla_x l\left(x,a^{(k-1)}(x)\right).
     \end{split}
\end{equation}
and
\begin{equation*}
\begin{split}
    \rho \lambda^{(k+1)}(x)=&D\lambda^{(k+1)}(x)g(x, a^{(k)}(x))\\
    &+D_x g\left(x,a^{(k)}(x)\right)\lambda^{(k)}(x)+\nabla_x l\left(x,a^{(k)}(x)\right).\\
\end{split}
\end{equation*} 
Then the difference is $\lambda^{(k+1)}-\lambda^{(k)}(x)$ is  
\begin{equation*}
\begin{split}
     &\rho(\lambda^{(k+1)}-\lambda^{(k)}(x))\\
     =&D\lambda^{(k+1)}(x)g(x,a^{(k)}(x))-D  \lambda^{(k)}(x)g(x,a^{(k-1)}(x))+D_x g(x,a^{(k)}(x))\lambda ^{(k)}(x)\\
     &-D_x g(x,a^{(k-1)}(x))\lambda^{(k-1)} (x)+\nabla_x l(x,a^{(k)}(x))-\nabla_x l(x,a^{(k-1)}(x))\\
     =&D\left(\lambda^{(k+1)}(x)-\lambda^{(k)}(x)\right)g(x,a^{(k)}(x))
     \\&+D\lambda^{(k+1)}(x)\left(g(x,a^{(k)}(x))-g(x,a^{(k-1)}(x))\right)
     \\
     &+\bigg(D_x g(x,a^{(k)}(x))-D_x g(x,a^{(k-1)}(x))
     \bigg)\lambda^{(k)}(x)+\nabla_x l(x,a^{(k)}(x))\\
     &-\nabla_x l(x,a^{(k-1)}(x))+D_x g(x,a^{(k-1)}(x))\left(\lambda^{(k)}(x)-\lambda^{(k-1)}(x)\right).
     \end{split}
\end{equation*}
We consider the error in  the following $L_\alpha^2$ sense with $\alpha>1$. 
Taking the inner product of  $\lambda^{(k+1)}-\lambda^{(k)}(x)$ with the previous expression, we have
\begin{equation*}
\begin{split}
 \rho e^{(k+1)} &\coloneqq     \rho \int_{\mathbb{R}^d}\frac{\left\|\lambda^{(k+1)}(x)-\lambda^{(k)}(x)\right\| ^2}{(1+\|x\|^2)^{2\alpha}} \d x\\
     \leq &
     \frac{1}{2}\bigg\|\int_{\mathbb{R}^d}D(\|\lambda^{(k+1)}(x)-\lambda^{(k)}(x)\|^2)\frac{g(x,a^{(k)}(x))}{(1+\|x\|^2)^{2\alpha}}\d x \bigg\|
    \\
    +&\bigg\|\int_{\mathbb{R}^d} D\lambda^{(k+1)}(x) \Big(g(x,a^{(k)})(x)-g(x,a^{(k-1)}(x))\Big)\frac{\lambda^{(k+1)}(x)-\lambda^{(k)}(x)}{(1+\|x\|^2)^{2\alpha}} \d x\bigg\|
    \\
    +&\bigg\|\int_{\mathbb{R}^d}\bigg(D_x g(x,a^{(k)}(x))-D_x g(x,a^{(k-1)}(x))\bigg)\lambda^{(k)}(x)\frac{\lambda^{(k+1)}(x)-\lambda^{(k)}(x)}{(1+\|x\|^2)^{2\alpha}}\d x
 \bigg\|   \\
    +&\bigg\|\int_{\mathbb{R}^d} \bigg(\nabla_x l(x,a^{(k)}(x))-\nabla_x l(x,a^{(k-1)}(x))\bigg)\frac{\lambda^{(k+1)}(x)-\lambda^{(k)}(x)}{(1+\|x\|^2)^{2\alpha}}\d x\bigg\|
     \\
    +&\bigg\|\int_{\mathbb{R}^d}D_x g(x,a^{(k-1)}(x))\frac{(\lambda^{(k)}(x)-\lambda^{(k-1)}(x))\left(\lambda^{(k+1)}(x)-\lambda^{(k)}(x)\right)}{(1+\|x\|^2)^{2\alpha}}  \d x\bigg\|\\
    :=&I_1+I_2+I_3+I_4+I_5.\\
    \end{split}
    \end{equation*}
   By Lemma \ref{lem:iter} and Assumption \ref{ass:A1},we have
 \begin{equation*}
   \|g(x,a^{(k)})\|  
    \leq \bar{g}(1+C_1) (1+\|x\|).
\end{equation*}
   Integration by part for the first term $I_1$  gives
   \begin{equation}
   \begin{split}
 I_1\leq & \frac{1}{2}\int_{\mathbb{R}^d}\frac{\|\lambda^{(k+1)}(x)-\lambda^{(k)}(x)\|^2}{(1+\|x\|^2)^{2\alpha}}
    \bigg[
    \left\|D_x g(x,a^{(k)}(x)) \right\|
    \\&+
    \left\|D_a g(x,a^{(k)}(x)\right\|
    \left\| D a^{(k)}(x)\right\| + 
    \frac{4\alpha \|x\|}{1+\|x\|^2}  \bar{g}(1+C_1)(1+\|x\|)  \bigg] \d x  \\
    \leq & \frac{1}{2}\int_{\mathbb{R}^d}\frac{\|\lambda^{(k+1)}(x)-\lambda^{(k)}(x)\|^2}{(1+\|x\|^2)^{2\alpha}}
    \bigg[
    \bigg\|D_x g (x,a^{(k)}(x))\bigg\|\\
    &+
   \bar{c}
    \left\| D a^{(k)}(x)\right\| + 
     5\alpha\bar{g}(1+C_1)\bigg] \d x 
     \\
     \leq & e^{(k+1)} \left(\bar{g}+\bar{c}C_2+5\alpha\bar{g}(1+C_1)\right),
     \end{split}\label{equ: I1}
\end{equation}using $ \frac{\|x\|(1+\|x\|)}{1+\|x\|^2}< \frac{5}{4}$ for $\forall x \in \mathbb{R}$ in the last second equation.\\ By the mean value theorem for $g(x,\cdot)$ and Lemma \ref{lem:iter}, the second term $I_2$ is
  \begin{equation*}
  \begin{split}
      I_2
      =&\int_{\mathbb{R}^d}\bigg\|D\lambda^{(k+1)}(x)
      D_a g\Big(x,a^{(k-1)}(x)+\delta_1 (x)(a^{(k)}(x)-a^{(k-1)}(x))\Big)\\& (a^{(k)}(x)-a^{(k-1)}(x))\frac{\lambda^{(k+1)}(x)-\lambda^{(k)}(x)}{(1+\|x\|^2)^{2\alpha}}\bigg\|\d x\\
    \leq &\bar{c}C_4 \int_{\mathbb{R}^d} \frac{\|\lambda^{(k+1)}(x)-\lambda^{(k)}(x)\|\|a^{(k)}(x)-a^{(k-1)}(x)\|} {(1+\|x\|^2)^{2\alpha}}\d x,
    \end{split}
\end{equation*}
where a function $\delta_1(x)$ is $\mathbb{R}^d\to \mathbb{R}$.\\ The third term $I_3=0$ because $D_x g(x,a)$ is independent of $a$.
The fourth term $I_4$ is 
   \begin{equation*}
        I_4 
    \leq  \bar{l}_{2}  \int_{\mathbb{R}^d} \frac{\|\lambda^{(k+1)}(x)-\lambda^{(k)}(x)\|\|a^{(k)}(x)-a^{(k-1)}(x)\|}{(1+\|x\|^2)^{2\alpha}}\d x.
 \end{equation*}
The last term $I_5$ is
\begin{equation}
\begin{split}\label{equ: I5}
I_5 \coloneqq  & \frac12
  \int_{\mathbb{R}^d}\frac{\left\|\lambda^{(k+1)}(x)-\lambda^{(k)}(x)\right\|^2}{(1+\|x\|^2)^{2\alpha}}
  \left\|D_x g(x,a^{(k-1)}(x))\right\| \d x\\&
+\frac12
  \int_{\mathbb{R}^d}\frac{ \left\|\lambda^{(k)}(x)-\lambda^{(k-1)}(x)\right\|^2}{(1+\|x\|^2)^{2\alpha}}
  \left\|D_x g(x,a^{(k-1)}(x))\right\| \d x
  \\
  \leq & \frac{\bar{g}}{2} \left(e^{(k+1)} +e^{(k)}\right)
  \end{split}
    \end{equation}
Next we estimate the bound of $\left\|a^{(k)}(x)-a^{(k-1)}(x)\right\|$ by $\left\|\lambda^{(k)}(x)-\lambda^{(k-1)}(x)\right\|$. By the first order necessary condition, we have 
\begin{equation*}
    \begin{split}
    0=&\nabla_a l(x,a^{(k)})+c^\top\lambda^{(k)}(x);\\
    0=&\nabla_a l(x,a^{(k-1)})+c^\top\lambda^{(k-1)}(x).
    \end{split}
\end{equation*}
Then, by the mean value theorem, there exist $\gamma^{(k+1)}_1(x):\mathbb{R}^d\to\mathbb{R}$ such that
\begin{equation*}
\begin{split}
      (\nabla_a\nabla_a^\top) l\left(x,a^{(k-1)}(x)+\gamma_1^{(k)}(x)(a^{(k)}-a^{(k-1)}(x))\right)&(a^{(k)}-a^{(k-1)}(x))
      \\&+  c^\top(\lambda^{(k)}(x)-\lambda^{(k-1)}(x))=0.
      \end{split}
\end{equation*}
Thus
\begin{equation*}
\begin{split}
    a^{(k)}(x)-a^{(k-1)}(x)=&-\bigg((\nabla_a\nabla^\top_a) l\left(x,a^{(k-1)}(x)+\gamma_1^{(k)}(x)(a^{(k)}-a^{(k-1)}(x))\right)\bigg)^{-1}\\&\cdot \bigg(c^\top (\lambda^{(k)}(x)-\lambda^{(k-1)}(x))\bigg).
    \end{split}
\end{equation*}
Since $\|(\nabla_a\nabla_a^\top) l(\cdot,\cdot)\|>c_s$, we have
$$\left\|a^{(k)}(x)-a^{(k-1)}(x)\right\|\leq \frac{\bar{c}}{c_s}\left\|\lambda^{(k)}(x)-\lambda^{(k-1)}(x)\right\|$$
and then
\begin{equation}
    \begin{split}\label{equ: I2}
    I_2\leq & \frac{\bar{c}^2C_4}{c_s}\int_{\mathbb{R}^d} \frac{\|\lambda^{(k+1)}(x)-\lambda^{(k)}(x)\|\|\lambda^{(k)}(x)-\lambda^{(k-1)}(x)\|}{(1+\|x\|^2)^{2\alpha}}\d x
    \leq  \frac{\bar{c}^2C_4}{2c_s} (e^{(k+1)}+e^{(k)})
\end{split}\end{equation}
and
\begin{equation}\label{equ: I4}
    \begin{split}
    I_4\leq & \bar{l}_{2}\int_{\mathbb{R}^d} \frac{\|\lambda^{(k+1)}(x)-\lambda^{(k)}(x)\|\|\lambda^{(k)}(x)-\lambda^{(k-1)}(x)\|}{(1+\|x\|^2)^{2\alpha}}\d x\leq  \frac{\bar{l}_{2}}{2} (e^{(k+1)}+e^{(k)})
\end{split}
\end{equation}
Combining \eqref{equ: I1}, \eqref{equ: I5}, \eqref{equ: I2}, \eqref{equ: I4}, we have
\begin{equation*}\begin{split}
\rho e^{(k+1)}\leq &e^{(k+1)} \left(\bar{g}+\bar{c}C_2+5\alpha\bar{g}(1+C_1)\right)+\\&\frac{\bar{c}^2C_4}{2c_s} (e^{(k+1)}+e^{(k)})+\frac{\bar{l}_{2}}{2} (e^{(k+1)}+e^{(k)})+\frac{\bar{g}}{2}(e^{(k+1)}+e^{(k)}).\end{split}\end{equation*}
Consequently, \begin{equation}\begin{split}
    e^{(k+1)}&\leq \frac{\frac{\bar{c}^2C_4}{2c_s}+\frac{\bar l_{2}}{2}+\frac{\bar{g}}{2}}{\rho-\left(\bar{g}+\bar{c}C_2+5\alpha\bar{g}(1+C_1)\right)-\frac{\bar{c}^2C_4}{2c_s}-\frac{\bar {l}_{2}}{2}-\frac{\bar{g}}{2}}e^{(k)}\coloneqq\eta e^{(k)}.\end{split}\end{equation}
Select $\rho_2$ to be 
\begin{equation}
\rho_2=\max\left\{\rho_1,2\bar{g}+\bar{c}C_2+5\alpha\bar{g}(1+C_1)+\frac{\bar{c}^2C_4}{c_s}+\bar{l}_{2}\right\}.
\end{equation}
then for $\rho>\rho_2$, we have $e^{(k+1)}\leq \eta e^{(k)}$ where $\eta\in (0,1)$. $e^{(k+1)}$ will converge to 0 as $k\to \infty$.
That is
\begin{equation}\lim_{k\to \infty}\int_{\mathbb{R}^d}\frac{\left \|\lambda^{(k+1)}(x)-\lambda^{(k)}(x)\right\| ^2}{(1+\|x\|^2)^{2\alpha}} \d x=0.\end{equation}
\end{proof}
Finally, we show that the sequence $\{\lambda^{(k)}\}$ does converge to the classical solution by the corollary below, the proof  is shown in the supplementary material.
\begin{corollary}\label{cor:converge}
If there exists a classical solution of PDE \eqref{eq:lam-pde}, then $\lambda^{(k)}$ converges to the solution in $L_{\alpha}^2$ sense.
\end{corollary}
\begin{proof}
 According to \cref{thm: main}, there exists $\lambda(x)\in L^2_\alpha$ such that $\lambda^k(x)\to\lambda (x)$ in $L_\alpha^2$ sense. Denote $$a(x)=\argmin_a [g(x,a)\cdot\lambda(x)+l(x,a)]$$
We then check that $\lambda(x)$ and $a(x)$ are the solutions for \eqref{eq:lam-pde} and \eqref{eq:opt-v-grad}. Integrate \eqref{equ: grad_pde} in $L_\alpha^2$ sense on both sides, and
let $\phi\in C^\infty_0 (\mathbb{R})$ be the test function, then we have 
\begin{equation}
\label{equ:int_pde_grad1}
\begin{split}
     \int_{\mathbb{R}^d}\frac{\rho \lambda^{(k)}(x)\phi(x)}{(1+\|x\|^2)^{2\alpha}}\d x=&\int_{\mathbb{R}^d} \frac{D\lambda^{(k)}(x)g\left(x,a^{(k-1)}(x)\right)\phi(x)}{(1+\|x\|^2)^{2\alpha}}\d x\\&+ \int_{\mathbb{R}^d} \frac{D_x g\left(x,a^{(k-1)}(x)\right)\lambda^{(k-1)}(x)\phi(x)}{(1+\|x\|^2)^{2\alpha}}\d x\\&+ \int_{\mathbb{R}^d} \frac{\nabla_x l\left(x,a^{(k-1)}(x)\right)\phi(x)}{(1+\|x\|^2)^{2\alpha}}\d x.
     \end{split}
\end{equation}
Consider \eqref{equ:int_pde_grad1} as $k\to \infty$, for the term on the left hand side, obviously
$$\int_{\mathbb{R}^d}\frac{\rho \lambda^{(k)}(x)\phi(x)}{(1+\|x\|^2)^{2\alpha}}\d x\to \int_{\mathbb{R}^d}\frac{\rho \lambda(x)\phi(x)}{(1+\|x\|^2)^{2\alpha}}\d x.$$
For the first term on the right hand side in \eqref{equ:int_pde_grad1}
\begin{equation*}
\begin{split}
    &\int_{\mathbb{R}^d}\frac{D\lambda^{(k+1)}(x)g(x,a^{(k)}(x))\phi (x)}{(1+\|x\|^{2\alpha})^{2\alpha}}\d x\\
=&\int_{\mathbb{R}^d} \lambda^{(k+1)}(x)D\bigg[\frac{g(x,a^{(k)}(x))\phi (x)}{(1+\|x\|^2)^{2\alpha}}\bigg]\d x\\
    =&\int_{\mathbb{R}^d} \lambda D\bigg[\frac{g(x,a(x))\phi (x)}{(1+\|x\|^2)^{2\alpha}}\bigg]\d x+\int_{\mathbb{R}^d} \lambda^{(k+1)}(x)D\bigg[\frac{(g(x,a^{(k)}(x))-g(x,a(x)))\phi (x)}{(1+\|x\|^2)^{2\alpha}}\bigg]\d x\\
    =&\int_{\mathbb{R}^d} \frac{D\lambda (x)g(x,a(x))\phi (x)}{(1+\|x\|^2)^{2\alpha}}\d x+\int_{\mathbb{R}^d} D\lambda^{(k+1)}(x)\bigg[\frac{\bar{c}(a^{(k)}(x)-a(x))\phi (x)}{(1+\|x\|^2)^{2\alpha}}\bigg]\d x\\
=&\int_{\mathbb{R}^d} \frac{D\lambda (x)g(x,a(x))\phi (x)}{(1+\|x\|^2)^{2\alpha}}\d x+\bar{c}\langle D\lambda^{(k+1)}(x)\phi(x), a^{(k)}(x)-a(x)\rangle_{L_\alpha^2}\\
\leq&\int_{\mathbb{R}^d} \frac{D\lambda (x)g(x,a(x))\phi (x)}{(1+\|x\|^2)^{2\alpha}}\d x+ \frac{\bar{c}^2}{c_s}\|D\lambda^{(k+1)}(x)\phi(x)\|_{L_\alpha^2}\|\lambda^{(k)}(x)-\lambda(x)\|_{L_\alpha^2}.\end{split}\end{equation*}
And $\|\lambda^{(k)}(x)-\lambda(x)\|_{L_\alpha^2}$ goes to $0$ when $k\to \infty$. So we have
$$\int_{\mathbb{R}^d}\frac{D\lambda^{(k+1)}(x)g(x,a^{(k)}(x))\phi (x)}{(1+\|x\|^{2\alpha})^{2\alpha}}\d x\to \int_{\mathbb{R}^d}\frac{D\lambda(x)g(x,a(x))\phi (x)}{(1+\|x\|^{2\alpha})^{2\alpha}}\d x$$in $L_\alpha^2$ sense. Similarly, for the second term, there holds
\begin{equation*}
\begin{split}
&\int_{\mathbb{R}^d} \frac{D_x g (x,a^{(k)}(x))\lambda^{(k-1)}(x)\phi (x)}{(1+\|x\|^{2\alpha})^{2\alpha}}\d x\\=& \int_{\mathbb{R}^d} \frac{D_x g (x,a(x))\lambda^{(k-1)}(x)\phi (x)}{(1+\|x\|^{2\alpha})^{2\alpha}}\d x\\&+\int_{\mathbb{R}^d} \frac{(D_x g (x,a^{(k)}(x))-D_x g (x,a(x)))\lambda^{(k-1)}(x)\phi (x) }{(1+\|x\|^{2\alpha})^{2\alpha}}\d x\\
\to &  \int_{\mathbb{R}^d} \frac{D_x g (x,a(x))\lambda(x)\phi (x)}{(1+\|x\|^{2\alpha})^{2\alpha}}\d x.
\end{split}
\end{equation*}
when $k\to \infty$. For the third term
\begin{equation*}
\begin{split}
&\int_{\mathbb{R}^d}\frac{\nabla_x l(x,a^{(k)}(x))\phi(x)}{(1+\|x\|^{2\alpha})^{2\alpha}}\d x\\
=&\int_{\mathbb{R}^d}\frac{\nabla_x (l(x,a^{(k)}(x))-l(x,a(x)))\phi(x)}{(1+\|x\|^{2\alpha})^{2\alpha}}\d x+\int_{\mathbb{R}^d}\frac{\nabla_x l(x,a(x))\phi(x)}{(1+\|x\|^{2\alpha})^{2\alpha}}\d x\\
\leq&\int_{\mathbb{R}^d}\frac{\bar{l}_2(a^k(x)-a(x))\phi(x)}{(1+\|x\|^{2\alpha})^{2\alpha}}\d x+\int_{\mathbb{R}^d}\frac{\nabla_x l(x,a(x))\phi(x)}{(1+\|x\|^{2\alpha})^{2\alpha}}\d x\\
\leq & \int_{\mathbb{R}^d}\frac{\frac{\bar{l}_2\bar{c}}{c_s}(\lambda^k(x)-\lambda(x))\phi(x)}{(1+\|x\|^{2\alpha})^{2\alpha}}\d x+\int_{\mathbb{R}^d}\frac{\nabla_x l(x,a(x))\phi(x)}{(1+\|x\|^{2\alpha})^{2\alpha}}\d x\\
\to & \int_{\mathbb{R}^d}\frac{\nabla_x l(x,a(x))\phi(x)}{(1+\|x\|^{2\alpha})^{2\alpha}}\d x
\end{split}
\end{equation*}
 when $k\to \infty$. As a result
\begin{equation}
\label{equ:int_pde_grad}
\begin{split}
     \int_{\mathbb{R}^d}\frac{\rho \lambda(x)\phi(x)}{(1+\|x\|^2)^{2\alpha}}\d x&=\int_{\mathbb{R}^d} \frac{D\lambda(x)g\left(x,a(x)\right)\phi(x)}{(1+\|x\|^2)^{2\alpha}}\d x\\&+ \int_{\mathbb{R}^d} \frac{D_x g\left(x,a(x)\right)\lambda(x)\phi(x)}{(1+\|x\|^2)^{2\alpha}}\d x+ \int_{\mathbb{R}^d} \frac{\nabla_x l\left(x,a(x)\right)\phi(x)}{(1+\|x\|^2)^{2\alpha}}\d x.
     \end{split}
\end{equation}
It shows $\lambda(x)$ and $a(x)$ are solutions for \eqref{eq:lam-pde} and \eqref{eq:opt-v-grad}, respectively.
\end{proof}

\section{Numerical Methods} 
\label{sec:nm}

Our algorithm is the  policy iteration based on $\lambda$
and it is clear that the main challenge  is to solve 
the system of linear PDEs \eqref{eq:iterlambda2}
in any dimension. It is worthwhile to point out 
 that each PDE in \eqref{eq:iterlambda2} 
 is the same type of PDE as the GHJE  \eqref{779}.
So, the Galerkin approximate approach can be also 
applied for these equations  in \eqref{eq:iterlambda2},
 but to directly aim for the high dimensional problems,
 we use the method of characteristics and the 
 supervised learning. 
 
 Specifically,
 we first consider a family of functions, such as neural networks,  $\widehat{\Phi}(x;\theta)$
 to 
 numerically represent the value function,
  where $\theta\in\Theta$ is the set of parameters.
The gradient-value function $\widehat{\lambda}(x;\theta)=\nabla_x\widehat{\Phi}(x;\theta)$
 is then computed by automatic differentiation instead of finite difference.
 Secondly,  in each policy iteration $k$, 
 we compute the characteristics by numerical integrating the 
 state dynamics and calculate the true value $\Phi^{(k+1)}$ and gradient-value functions
 $\lambda^{(k+1)}$
 on the  characteristics curves
 based on the  PDE \eqref{779} and \eqref{eq:iterlambda2}.
 Then  these labelled data $(X(t), \Phi(X(t)), \lambda(X(t)))$
are fed into the supervised learning protocol by minimizing the mean squared error.
to  find the optimal $\theta^{(k+1)} $. 

In sequel, we discuss the details of method of characteristics 
on solving the PDEs \eqref{779} and \eqref{eq:iterlambda2}
on   characteristics curves.
We drop the {\bf PI-lambda} iteration index $k$ in this section
for notational ease. 

\subsection{ Method of characteristics}
Bearing in mind the similar form of  \eqref{779} and \eqref{eq:iterlambda2}
which are both  hyperbolic linear PDEs with the same advection,  
we   consider a general discussion. 
 Given a  control function $a(\cdot)$, 
 we denote  
 $G(x)=g(x,a(x))$ and define  $X(t)$ as the characteristic curve satisfying the following ODE with an arbitrary initial state $X_0 \in \mathbb{R}^d$:
\begin{equation}
\begin{cases}
{\d}X(t)=G(X) dt,\label{eq:4Aode}\\
X(0)=X_0.
\end{cases}
\end{equation}
We consider the following  PDE of the function $v$ 
\begin{equation}
\label{eq:generic1}
\rho v(x) - Dv(x) \cdot G(x) = R(x)
\end{equation}
where the source term $R$ is given.
Note that  \eqref{779} and \eqref{eq:iterlambda2} are special cases of \eqref{eq:generic1}
with different $R$ terms. Along the characteristic curve $X(t)$,  
by \eqref{eq:4Aode}  and \eqref{eq:generic1} we derive that 
\begin{equation*}
\begin{split}
\frac{\d}{\dt}\left[e^{-\rho t}v(X(t))\right]=-\rho e^{-\rho t}  v (X(t))+  e^{-\rho t}Dv(X(t))\cdot \frac{\d X}{\dt}=-  e^{-\rho t}R(X(t)).
\end{split}    
\end{equation*}
After taking integral in time,\begin{equation}
\begin{split}
\lim_{s\to +\infty}e^{-\rho s}v (X(s)) - e^{-\rho t}v(X(t))
=\int^{+\infty}_t  -e^{-\rho \tau} R(X(\tau))  \d \tau
\label{4Afirst}
\end{split}
\end{equation}
As time $s$ tends to infinity, suppose $\rho$ is large enough, we have
\[ 
v(X(t)) =e^{\rho t}\int_t^{+\infty} e^{-\rho \tau} R(X(\tau))\d \tau.
\]



\subsection{ Compute   the value function and the gradient
on the characteristics}
We apply the above method of characteristics to compute   the value function $\Phi$ and the gradient $\lambda = \nabla \Phi$.
 For the value function   in equation  \eqref{lin_HJE}, the $R$ function 
in \eqref{eq:generic1} is $l(x,{a}(x))$.
Then $\Phi$ in  \eqref{lin_HJE} has the values on $X(t)$:
\begin{equation}
\Phi (X(t)) = \int_t^{+\infty} e^{-\rho (\tau-t)} l(X(\tau), a(X(\tau)))\, \d \tau.
\end{equation}
 For   $\lambda^{(k+1)}$ in \eqref{eq:iterlambda2}, for each component $i$,
 $R(x)$ in \eqref{eq:generic1} now  refers to the right hand side  in  function  \eqref{eq:iterlambda2}, then 
  \begin{equation}
\lambda^{(k+1)}_i(x(t  ))
=
 \int_t ^{+\infty} e^{-\rho (\tau-t)}  
 r^{(k)}_i(\tau)\, \d \tau.
\end{equation}
where $$ r^{(k)}_i(\tau)= \sum_n   \frac{\partial g_n}{\partial{x_i}}  \lambda^{(k)}_n(X(\tau))+   \frac{\partial l }{\partial{x_i}}(X(\tau), {a}^{(k)}(X(\tau))).$$

\subsection{Supervised learning: interpolate the characteristic curve to the whole space}

With a  
characteristic curve
$X(\cdot)$ computed from 
\eqref{eq:4Aode}, we can obtain the value of  the value function $\Phi$ and the gradient $\lambda_i=\frac{\partial \Phi}{\partial x_i}$, $i=1,\ldots, d$,  along $X(t)$  {\it simultaneously}. By running multiple characteristic curves  starting from a set of the initial points
$\{X_0^{(n)}$,$1\leq n\leq N\}$ which are generally sampled uniformly,
 we obtain   a collection of observations  of $\Phi(X^{(n)}(t))$ and $\lambda (X^{(n)}(t))$  
 on these characteristics trajectories   $\set{X^{(n)}(t):  t\geq 0, 1\leq n\leq N }$. In practice, the continuous  path $X^{(n)}(t)$ 
is represented  by a finite number of ``images'' on the curve and these images on each curve
are chosen to have the roughly equal distance to each neighbouring image.  
 
 To interpolate the labelled data from the computed curves to the whole space,
 a family of approximate functions  $ \widehat{\Phi}(x;\theta)$ should be proposed   first by the users, which
 could  be  Galerkin form of  basis functions,  radial basis functions or   neural networks, etc. 
Then the  parameters  $\theta$ is found by minimizing the following  loss function
$ L(\theta)$   combining  two  mean square errors: 
\begin{equation}
\label{def:loss-fun}
\begin{split}
 L(\theta) =  &\mu  \sum_{n=1}^N \int \left\|\Phi(X^{(n)}(t))-\widehat{\Phi}_{\theta}(X^{(n)}(t))\right\|^2 \dt
 \\
&+(1-\mu)   \sum_{n=1}^N \int \left\| \lambda(X^{(n)}(t))-\nabla\widehat{\Phi}_{\theta}(X^{(n)}(t))\right\|^2 \dt
\end{split}
\end{equation}
where $0\leq\mu\leq 1$ is a factor 
to balance the loss from the value function and the gradient.
$\|\cdot\|$ is the Euclidean norm in $\Real^d$.  The gradient $\nabla \widehat{\Phi}_{\theta}$ is the gradient 
w.r.t. the state variable $x$ and computed by automatic differentiation.
  The training process of the models is to minimize the loss function 
\eqref{def:loss-fun} w.r.t.  $\theta$ by some standard  gradient-descent optimization methods such as 
ADAM \cite{kingma2014adam}.

A few remarks are discussed now to explain our
practical algorithm more clearly.
\begin{itemize}
    \item 
    Our algorithmic framework is the policy iteration based on $\lambda$. So 
    the computation of the data points on the characteristics curves and the training of 
the loss    \eqref{def:loss-fun} are performed at each policy iteration $k$.
One can adjust the number of characteristic trajectories $N$ and the number of training steps 
(the steps within the minimization procedure for the loss function).  
The trajectory number $N$ determines the amount of data and the training step determines the accuracy of supervised learning.

\item The loss \eqref{def:loss-fun} simply writes the contribution from each trajectory 
in the  continuous $L_2$ integration in time. 
Practically, this integration is represented by the sum from each discrete point on the curves.
For better fitting of the function  $\widehat{\Phi}_\theta$,  these points  are not supposed to  correspond to equal step size in time variable but should be arranged to   spread out evenly in space.
There are many practical ways to achieve this target such as using the arc-length parametrization
or setting a small ball as the  forbidden region for each prior point.
Our numerical tests use the arc-length parametrization for each trajectory.   
    \item 
    The choice of the initial 
    states $\set{X_0^{(n)}: 1\leq n\leq N}$
 can affect how the corresponding characteristics curves behave 
 in the space and we hope these finite number of curves can explore the space efficiently.
Some adaptive ideas are worth a try in practice. For example, 
more points may be sampled where the residual of
HJE is larger. However, 
since the whole characteristics curves nonlinearly depend on the initial, 
we use the uniform distribution  
in our numerical tests  for simplicity.
 \end{itemize}


\section{Numerical Examples}
\label{sec:nex}
This section presents the  numerical experiments to 
show   the advantage of our new method of  the policy iteration using $\lambda$  and $\Phi$  
over the method only using $\Phi$. We test three problems in all:  Linear-quadratic problem, Cart-pole balancing task and Advertising process.

\subsection{Linear-quadratic problem}

The control problem to be solved is a $d$-dim linear-quadratic case with the cost function
$$
\begin{aligned}
J(u)=&\int_0^\infty e^{-\rho t}( \|x(t)\|^2+\|u(t)\|^2)\text{d} t
\end{aligned}
$$subject to the dynamic system
$$\begin{aligned}\dot{x}&=Ax+Bu, \quad x(0)=x_0.\end{aligned}$$
Instead of solving the Riccati equation for this problem,
we apply  our method in Section \ref{sec:nm}
by using the network structure  
$$\hat{\Phi}_Q(x)=\frac{1}{2}x^\top(Q^\top+Q)x$$
for simplicity where $Q$ is the parameter to be determined, 
since we know the true value function is a quadratic function.
This type of parametrization can eliminate the approximation error
since the true value function belong to this family of parametrized functions. 

We apply the algorithm to the following three  choices of $A$ with $d=5$,  $B=I_d$
and $\rho=1$.
\begin{itemize}
    \item Test 1: $A=I_d$ where $I_d$ is the $d$-dim  identical matrix.
    \item  Test 2 :
$A=(a^Ta+I_d)/d  $ where $a$ is an $d$-by-$d$ matrix.
 Every component of $a$ is i.i.d. random variables sampled from standard normal distribution. 
\item Test 3: The setting of Test 3 is the same as Test 2 with a different realization of  $A$.

\end{itemize}
In our numerical tables,
 ``T1", ``T2" and ``T3"  refer to Test 1, Test 2 and Test 3 defined above, respectively.

We compute the value function in the box $[-1,1]^d$. The initial values of the characteristics $X_0^{(n)}$ are uniformly
sampled from this box. Only the labelled data on the trajectories inside the box are used to train the model $\widehat{\Phi}_Q$.
The training process to minimize the loss $L(\theta)$
uses the full-batch ADAM \cite{kingma2014adam}.

We measure  the accuracy of the numerical solution $\widehat{\Phi}_Q$ by the average residual of HJB equation of $N_p=10000$ points uniformly selected from $[-1,1]^d$:
\begin{align}\label{equ:DQ}
error=\frac{1}{N_p}\sum^{N_p}_{j=1}
\big\|
\rho \hat{\Phi}_Q(x^{(j)})
-&g\left(x^{(j)}, a^*(x^{(j)})\right) \cdot \nabla \hat{\Phi}_Q(x^{(j)})-l\left(x^{(j)}, a^*(x^{(j)})\right)
\big\|
\end{align}  
where  $a^*(x)=-B^\top\nabla \hat{\Phi}_Q(x)$.

We conduct two experiments on each of the above three tests
for different purposes to benchmark and understand  our algorithms.

{\bf Experiment 1}.
In Experiment 1, we study how insufficient amount of  
characteristics data will affect the accuracy.
Specifically, we change the number of the characteristic trajectories  $N$ between $2$ and $10$
while keeping all other settings the same. Fewer trajectories mean less amount of labelled data from the method of characteristics. 
At each policy iteration,
  the training for the supervised learning 
  to minimize the loss $L(\theta)$ takes a fixed number of 1000 ADAM steps  or 
reaches a prescribed low tolerance. 
The number of policy iterations is fixed as 30.

\begin{table}[]
\scriptsize
    \centering
    \begin{tabular}{m{0.5cm}<{\centering} | m{1cm}<{\centering}|m{1.cm}<{\centering}m{1.cm}<{\centering}m{1.cm}<{\centering}m{1.cm}<{\centering}m{1.cm}<{\centering}}
\hline\hline \multirow{2}{0.5cm}{}&  \multirow{2}{1cm}{\centering $\mu$} & \multicolumn{5}{c}{The number of characteristics trajectories $N$} \\
\cline{3-7}& & 2 &  4 & 6 & 8& 10  \\ \hline\hline
	\multirow{6}*{\centering T1} &$1.0$&\textit{Diverge}  & \textit{Diverge}  & \textit{Diverge} & \textit{0.0251}
&\textit{0.0080}
\\
 &$0.8$ &0.0382&0.0069& 0.0032
 &0.0027
  &0.0024\\
&$0.6$ & 0.0251 &0.0056  & 0.0022 & 0.0018 &0.0016\\
& $0.4$ & \textbf{0.0088} &  0.0041&  0.0020& 0.0016 & 0.0019\\
& $0.2$ &0.0017  & 0.0030 & \textbf{0.0015} & 0.0019 &0.0014\\
& $0.0$ & 0.0106 & \textbf{0.0026} & 0.0022 & \textbf{0.0013} & \textbf{0.0012}\\\hline\hline
	\multirow{6}*{\centering T2} &  $1.0$&2.9116& \textit{Diverge} & \textit{0.0860} & \textit{0.0281} & 0.0112
\\
& $0.8$& 0.0360 & \textbf{0.0097} & 0.0049 & 0.0060 & 0.0058
\\
&$0.6$ & 0.0370 & 0.0128 & \textbf{0.0046} & 0.0058 & \textbf{0.0045}\\
& $0.4$ & \textbf{0.0193} & 0.0204 & 0.0140 & \textbf{0.0044} & 0.0057\\
& $0.2$ & 0.0280 & 0.0193 & 0.0220 & 0.0198 & 0.0094  \\
& $0.0$ & \textit{Diverge} & \textit{Diverge} & 0.0358 & 0.0085 & \textit{0.0413} \\\hline\hline
	\multirow{6}*{\centering  T3} &  $1.0$&6.3956&\textit{1.3894}&0.1372&0.0262&0.0205
\\
& $0.8$& \textbf{0.0544}& \textbf{0.0269}&0.0259&\textbf{0.0153}&0.0120
\\
& $0.6$ & 0.1079 & 0.0365 & \textbf{0.0236} & 0.0162 & \textbf{0.0081}\\
& $0.4$ & 0.0806 & 0.0833& 2.6797 & 0.1816 & 0.0203 \\
& $0.2$ & \textit{Diverge} & 0.0754 & 0.2794 & 31.3591 & 0.0481 \\
& $0.0$ & \textit{Diverge} & 0.1773 & \textit{Diverge} & \textit{Diverge} & \textit{0.0834} \\\hline\hline
\end{tabular}
    \caption{Error (HJB residual) for various $\mu$ when the number of trajectories $N$ changes. ``T1", ``T2" and ``T3" refer to the three tests in the text.}
    \label{tab:lqr_trajectorynumber} 
\end{table}

Table \ref{tab:lqr_trajectorynumber} shows the results when $\mu$ varies
for each test. For each given $N$,
the collection of  $N$ initial states 
are the same at different $\mu$ for consistent comparison.
If the numerical value of $\Phi$ goes to infinity, we mark ``Diverge" in the table.  Otherwise, the average residual errors defined in \eqref{equ:DQ} of the last 20 iterations is reported.   For each setting,
 the best residual is highlighted in bold symbols and the worst residual (including the diverge case) is emphasised in italics.  
Form   Table \ref{tab:lqr_trajectorynumber}, we can see for all three tests,
$\mu=1$ (only using the value) or $\mu=0$ (only use the gradient-value) 
has the worst performance and may diverge in many cases,
while the loss corresponding to  $\mu$ strictly between $0$ and $1$ 
can achieve the best accuracy and we do not see divergence at all. 
However,  the value $\mu$ corresponding to 
the best accuracy result changes from test to test.
This table also confirms that with the increasing number $N$ of characteristics,
the final accuracy of the numerical value functions always 
gets better and better since more labelled data are provided.

\begin{figure*}[htbp]
\centering
\subfigure[Test 1 (Trajectory: 8)]{
\begin{minipage}[t]{0.32\linewidth}
\centering
\includegraphics[width=\textwidth]{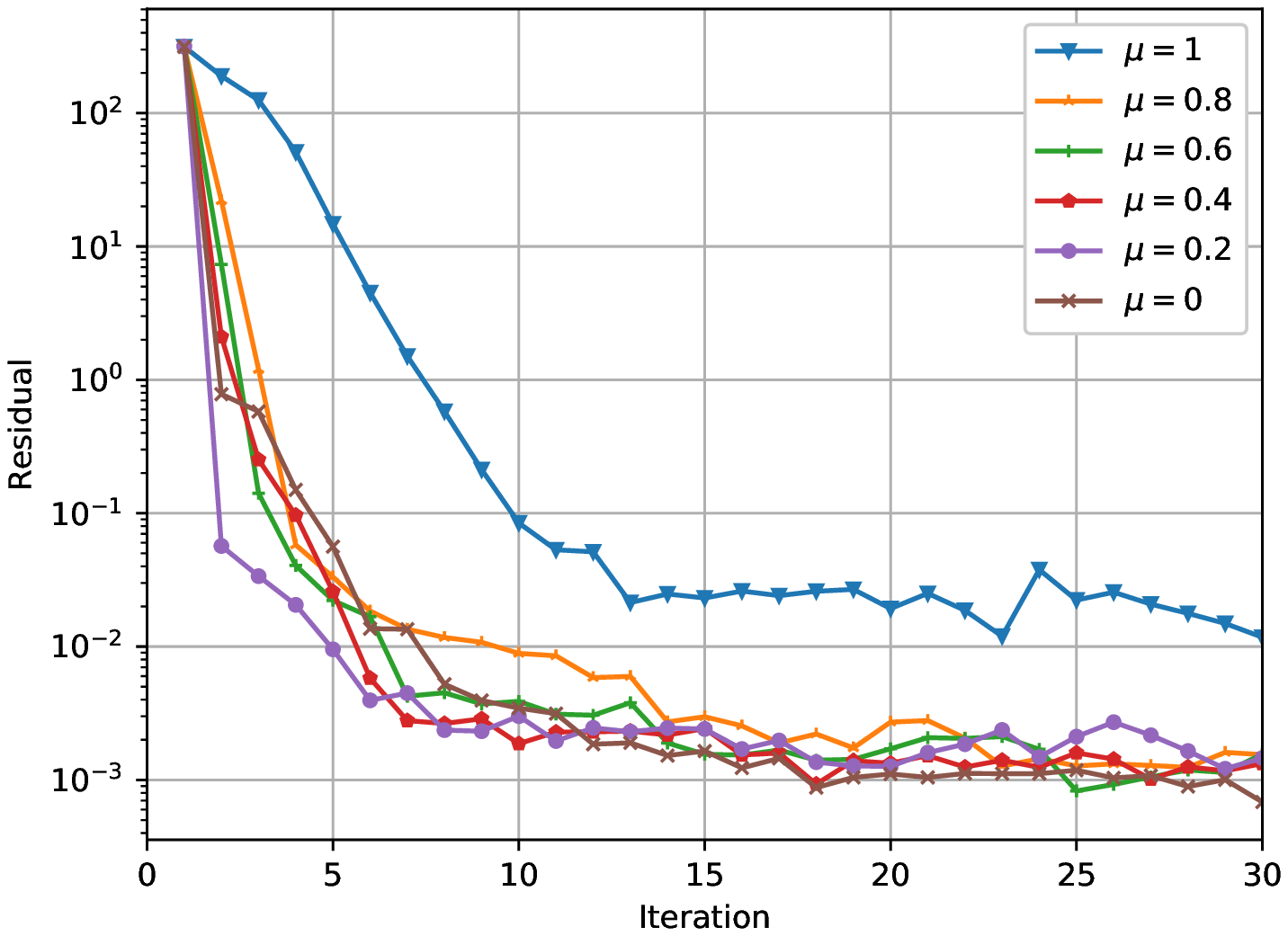}
\end{minipage}%
}%
\subfigure[Test 1 (Trajectory: 10)]{
\begin{minipage}[t]{0.32\linewidth}
\centering
\includegraphics[width=\textwidth]{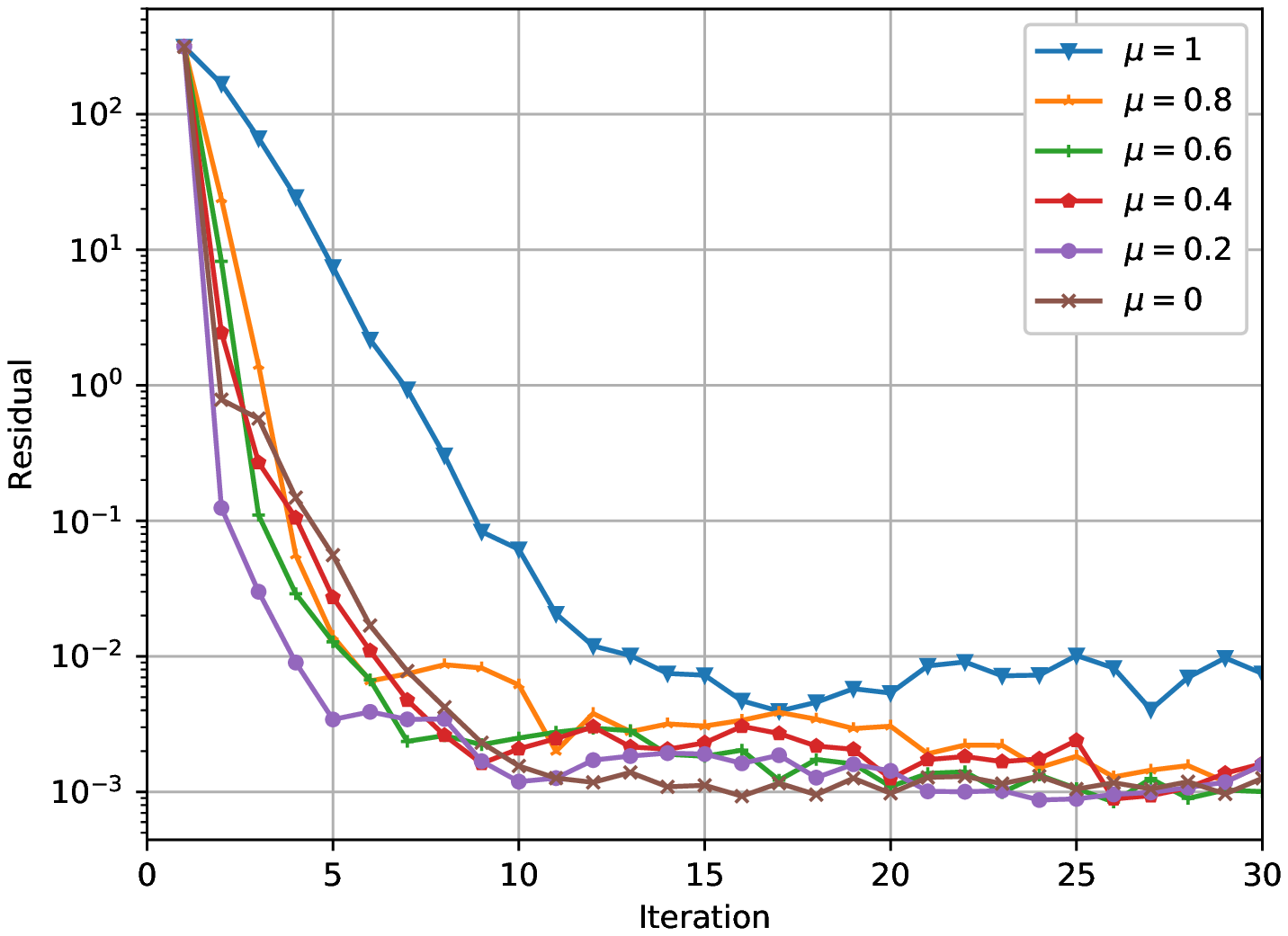}
\end{minipage}%
}%
\subfigure[Test 2 (Trajectory: 6)]{
\begin{minipage}[t]{0.32\linewidth}
\centering
\includegraphics[width=\textwidth]{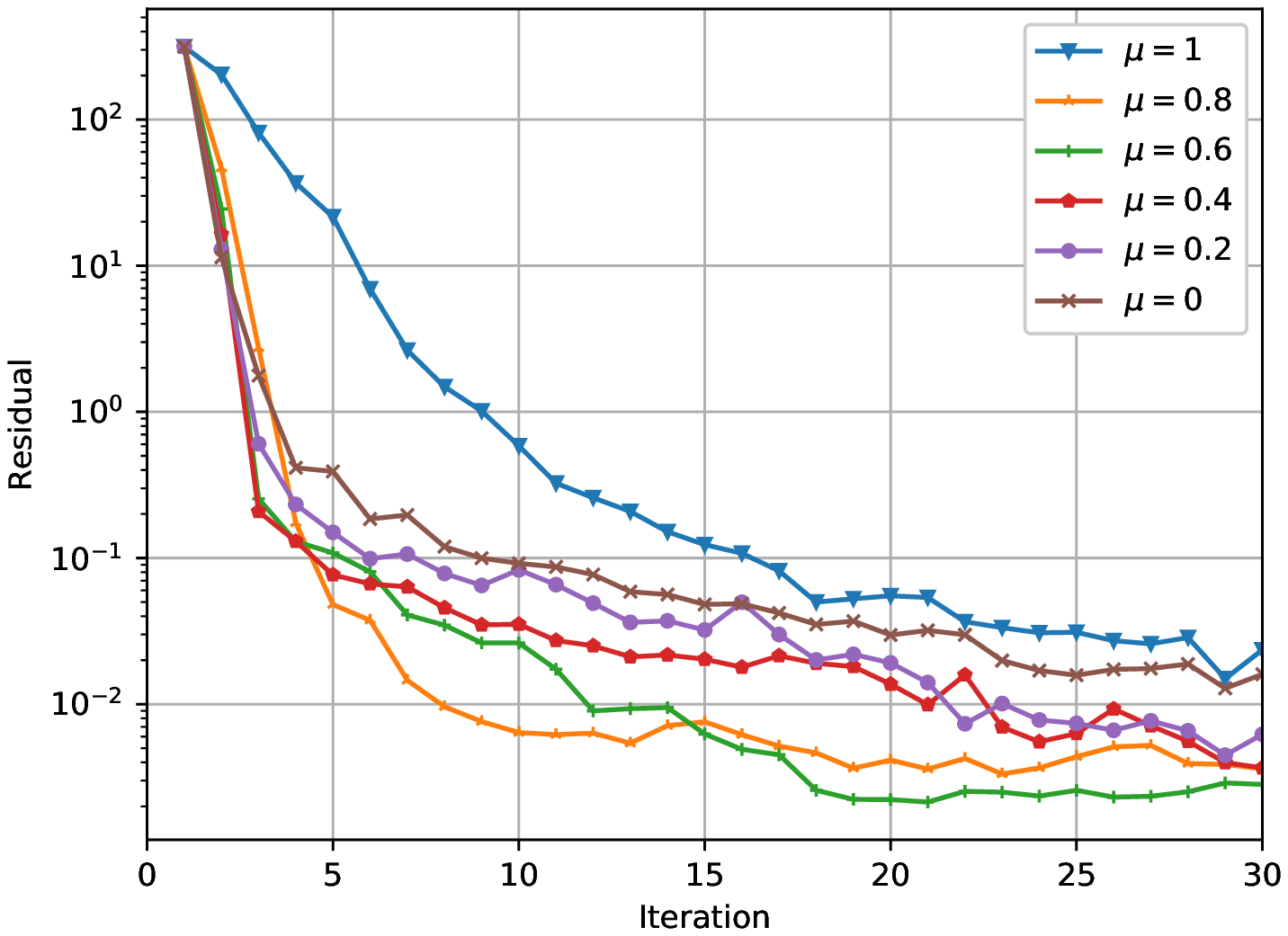}
\end{minipage}
}%

\subfigure[Test 2 (Trajectory: 8)]{
\begin{minipage}[t]{0.32\linewidth}
\centering
\includegraphics[width=\textwidth]{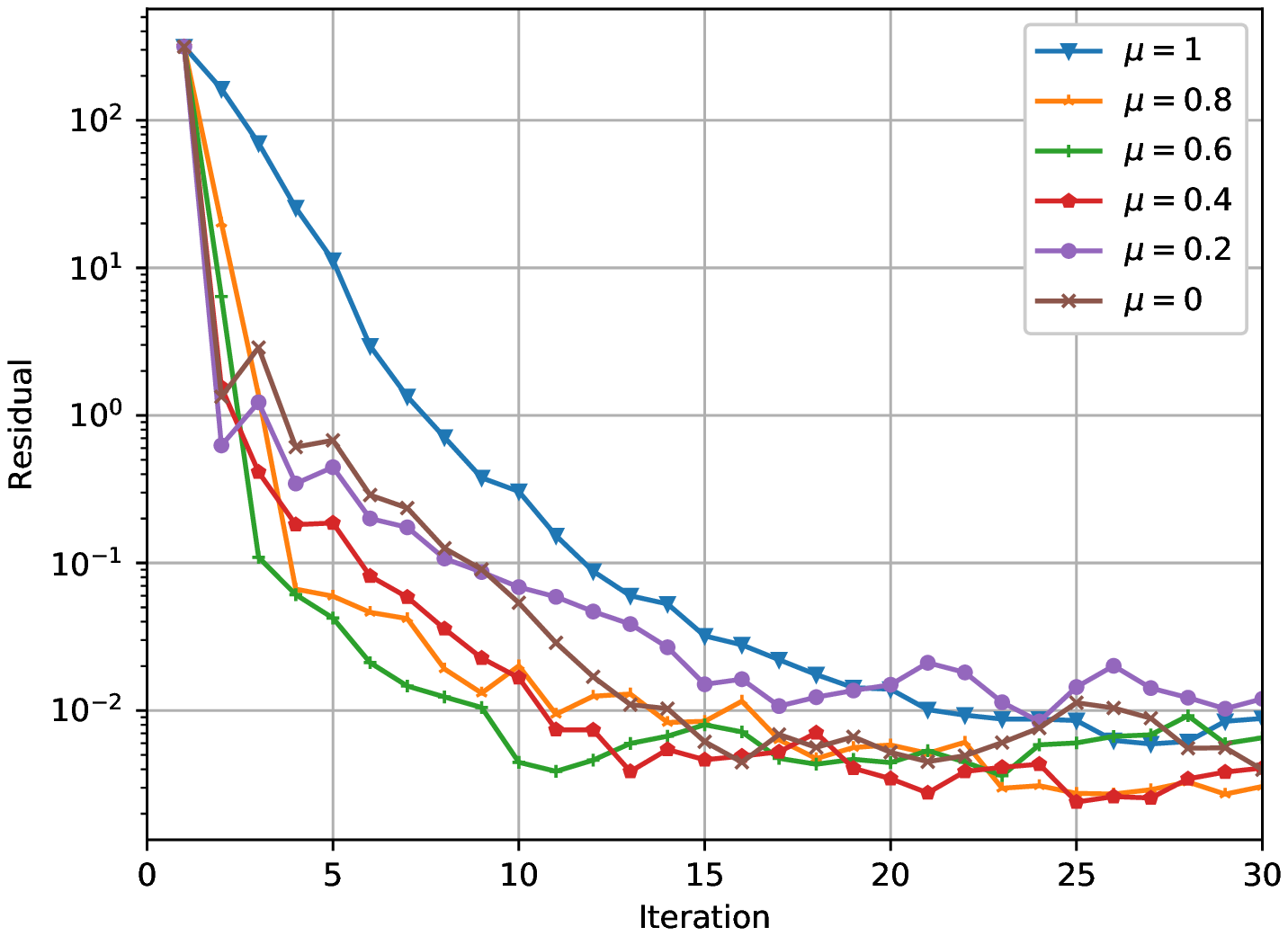}
\end{minipage}%
}%
\subfigure[Test 3 (Trajectory: 4)]{
\begin{minipage}[t]{0.32\linewidth}
\centering
\includegraphics[width=\textwidth]{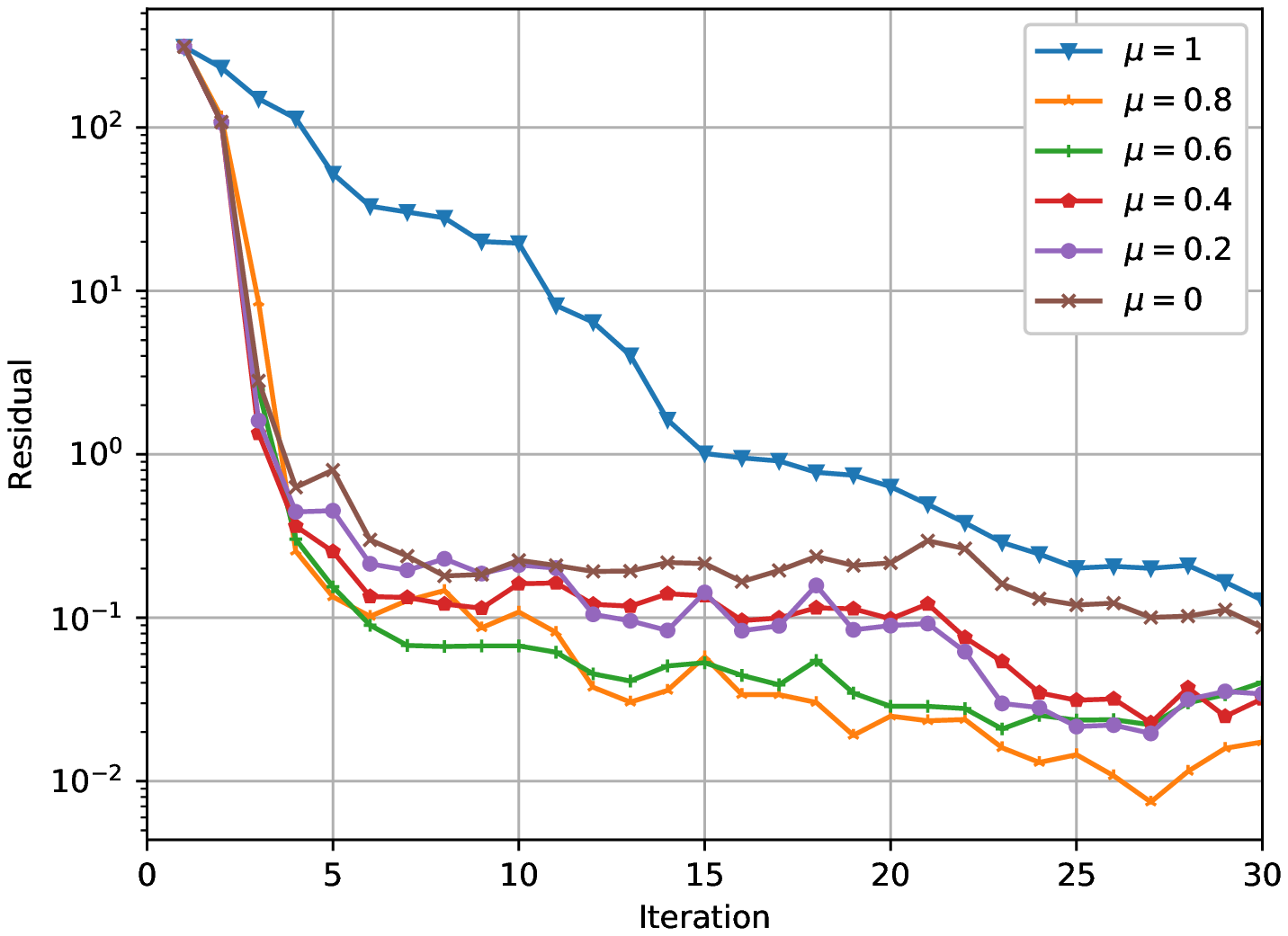}
\end{minipage}%
}%
\subfigure[Test 3 (Trajectory: 10)]{
\begin{minipage}[t]{0.32\linewidth}
\centering
\includegraphics[width=\textwidth]{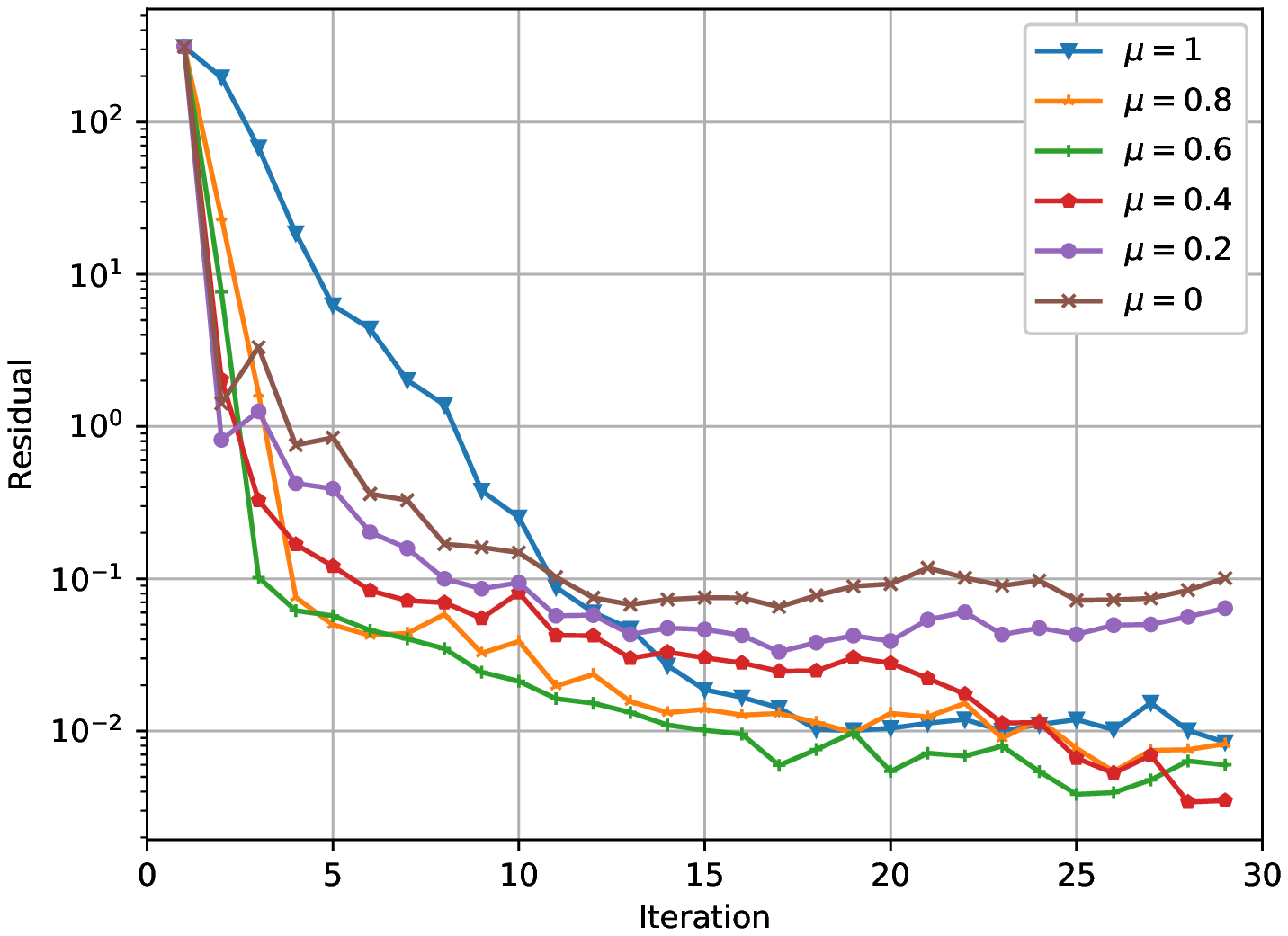}
\end{minipage}
}%

\caption{Error (HJB residual) vs policy iteration for various $\mu$.}
\label{fig:lqr_tra}
\end{figure*}


 {\bf Experiment 2}.
The purpose of  Experiment 2 is to test 
the performance of the methods  when the training process is not exact. 
Recall that in Experiment 1, we have set the maximum steps in training process 
as a sufficiently large number 1000. Here,  we limit this maximum training step 
 to  the range $10\sim200$.  
A small maximum training step  means less accuracy in fitting the value function. 
For each test, the algorithm is run up to  120 policy iterations and
the  number of  characteristics trajectories is fixed as a relatively small number $N=5$ now.

\begin{table}[htbp]
\centering
\scriptsize
\begin{tabular}{m{0.5cm}<{\centering}|m{1cm}<{\centering}|m{1.48cm}<{\centering}m{1.48cm}<{\centering}m{1.48cm}<{\centering}m{1.48cm}<{\centering}m{1.48cm}<{\centering}}
\hline\hline \multirow{2}{0.5cm}{}&  \multirow{2}{1cm}{\centering $\mu$} & \multicolumn{5}{c}{Train step} \\
\cline{3-7}
    & & 10 &  50 & 100 & 150& 200  \\ \hline\hline
 \multirow{6}*{T1}&$1.0$& \textit{1.476} & \textit{Diverge} & \textit{Diverge} & \textit{1.03$\times 10^{-2}$}& \textit{1.55$\times 10^{-2}$}
\\
 &$0.8$ &\textbf{4.36}$\boldsymbol{\rm \times 10^{-4}}$& \textbf{4.46$\boldsymbol{\rm \times 10^{-4}}$} & \textbf{4.58$\boldsymbol{\rm \times 10^{-4}}$}& \textbf{4.57$\boldsymbol{\rm \times 10^{-4}}$}&\textbf{4.51$\boldsymbol{\rm \times 10^{-4}}$}\\
&$0.6$ & 6.12$\rm \times 10^{-4}$& 6.18$\rm \times 10^{-4}$ & 6.25$\rm \times 10^{-4}$ & 6.25$\rm \times 10^{-4}$ & 6.45$\rm \times 10^{-4}$\\
& $0.4$ & 7.02$\rm \times 10^{-4}$& 7.10$\rm \times 10^{-4}$ & 7.14$\rm \times 10^{-4}$ & 7.12$\rm \times 10^{-4}$&  7.30$\rm \times 10^{-4}$\\
& $0.2$ & 7.62$\rm \times 10^{-4}$ & 7.66$\rm \times 10^{-4}$ & 7.68$\rm \times 10^{-4}$ & 7.70$\rm \times 10^{-4}$& 7.84$\rm \times 10^{-4}$\\
& $0.0$ & 8.04$\rm \times 10^{-4}$ & 8.02$\rm \times 10^{-4}$ & 8.08$\rm \times 10^{-4}$ & 8.10$\rm \times 10^{-4}$ & 8.19$\rm \times 10^{-4}$\\\hline\hline
\multirow{6}*{T2}& $1.0$& \textit{0.146} & \textit{Diverge} &\textit{Diverge} & \textit{Diverge} & \textit{Diverge}
\\
& $0.8$ & \textbf{2.93$\boldsymbol{\rm \times 10^{-4}}$} & \textbf{2.80$\boldsymbol{\rm \times 10^{-4}}$} & \textbf{2.84$\boldsymbol{\rm \times 10^{-4}}$} & \textbf{2.88$\boldsymbol{\rm \times 10^{-4}}$} & \textbf{2.92$\boldsymbol{\rm \times 10^{-4}}$} \\
&$0.6$ & 4.13$\rm \times 10^{-4}$ & 3.88$\rm \times 10^{-4}$ & 3.83$\rm \times 10^{-4}$ & 3.83$\rm \times 10^{-4}$ & 3.82$\rm \times 10^{-4}$\\
& $0.4$ & 4.42$\rm \times 10^{-4}$ & 4.51$\rm \times 10^{-4}$ & 4.46$\rm \times 10^{-4}$ & 4.46$\rm \times 10^{-4}$ & 4.36$\rm \times 10^{-4}$\\
& $0.2$ & 4.56$\rm \times 10^{-4}$ & 4.87$\rm \times 10^{-4}$ & 4.57$\rm \times 10^{-4}$ & 4.70$\rm \times 10^{-4}$& 4.76$\rm \times 10^{-4}$\\
& $0.0$ & 4.83$\rm \times 10^{-4}$& 5.15$\rm \times 10^{-4}$ & 4.98$\rm \times 10^{-4}$ & 5.07$\rm \times 10^{-4}$ & 2.23$\rm \times 10^{-3}$\\\hline\hline
\multirow{6}*{T3}&  $1.0$& \textit{7.47$\times 10^{-2}$} & \textit{Diverge} & 5.91$\rm \times 10^{-4}$& 8.48$\rm \times 10^{-4}$ & \textit{1.21$\times 10^{-2}$}
\\
& $0.8$ & 3.53$\rm \times 10^{-4}$ & \textbf{2.32$\boldsymbol{\rm \times 10^{-4}}$} & \textbf{2.36$\boldsymbol{\rm \times 10^{-4}}$} &2.45$ \rm \times 10^{-4}$ & \textbf{2.61$\boldsymbol{\rm \times 10^{-4}}$}\\
&$0.6$ & \textbf{3.39$\boldsymbol{\rm \times 10^{-4}}$}& 3.06$\rm \times 10^{-4}$ & 3.17$\rm \times 10^{-4}$ & 3.35$\rm \times 10^{-4}$ & 3.19$\rm \times 10^{-4}$\\
& $0.4$ & 4.31$\rm \times 10^{-4}$ & 3.59$\rm \times 10^{-4}$ & 3.56$\rm \times 10^{-4}$ & \rm 3.65$\rm \times 10^{-4}$ &\rm 7.77$\rm \times 10^{-4}$\\
& $0.2$ & 4.53$\rm \times 10^{-4}$ & 3.96$\rm \times 10^{-4}$ & 4.07$\rm \times 10^{-4}$ & \textbf{1.08$\boldsymbol{\rm \times 10^{-4}}$} & 1.05$\rm \times 10^{-2}$\\
& $0.0$ & 5.09$\rm \times 10^{-4}$ & 4.20$\rm \times 10^{-4}$& \textit{8.48$\times 10^{-3}$} & \textit{5.20$\times 10^{-3}$} & 8.31$\rm \times 10^{-3}$\\ \hline\hline
\end{tabular}
    \caption{ Error (HJB residual) for various $\mu$  when the training steps change.}
    \label{tab:lqr_trainstep}
\end{table}

The average HJB residuals  of the last 20 policy iterations are 
reported  in Table \ref{tab:lqr_trainstep} to measure the accuracy.
 This  table shows that $\mu=1$ has the worst performance in Test 1 and Test 2
and  neither $\mu=1$ nor $\mu=0$ can perform well in Test 3. It is confirmed that 
the setting of $\mu$ strictly between $0$ and $1$
is more robust to incomplete training and also has better performance in accuracy.
We can also see from this table that there is in general no necessity to 
use strict stopping criteria for training the interpolation for $\widehat{\Phi}_Q$.
Even a small training step $10$ with a choice  $\mu\in (0,1)$ can
have the same final accuracy   as the large training step $200$.

\begin{figure*}[htbp]
\centering
\subfigure[Test 1 (Train step: 10)]{
\begin{minipage}[t]{0.32\linewidth}
\centering
\includegraphics[width=\linewidth]{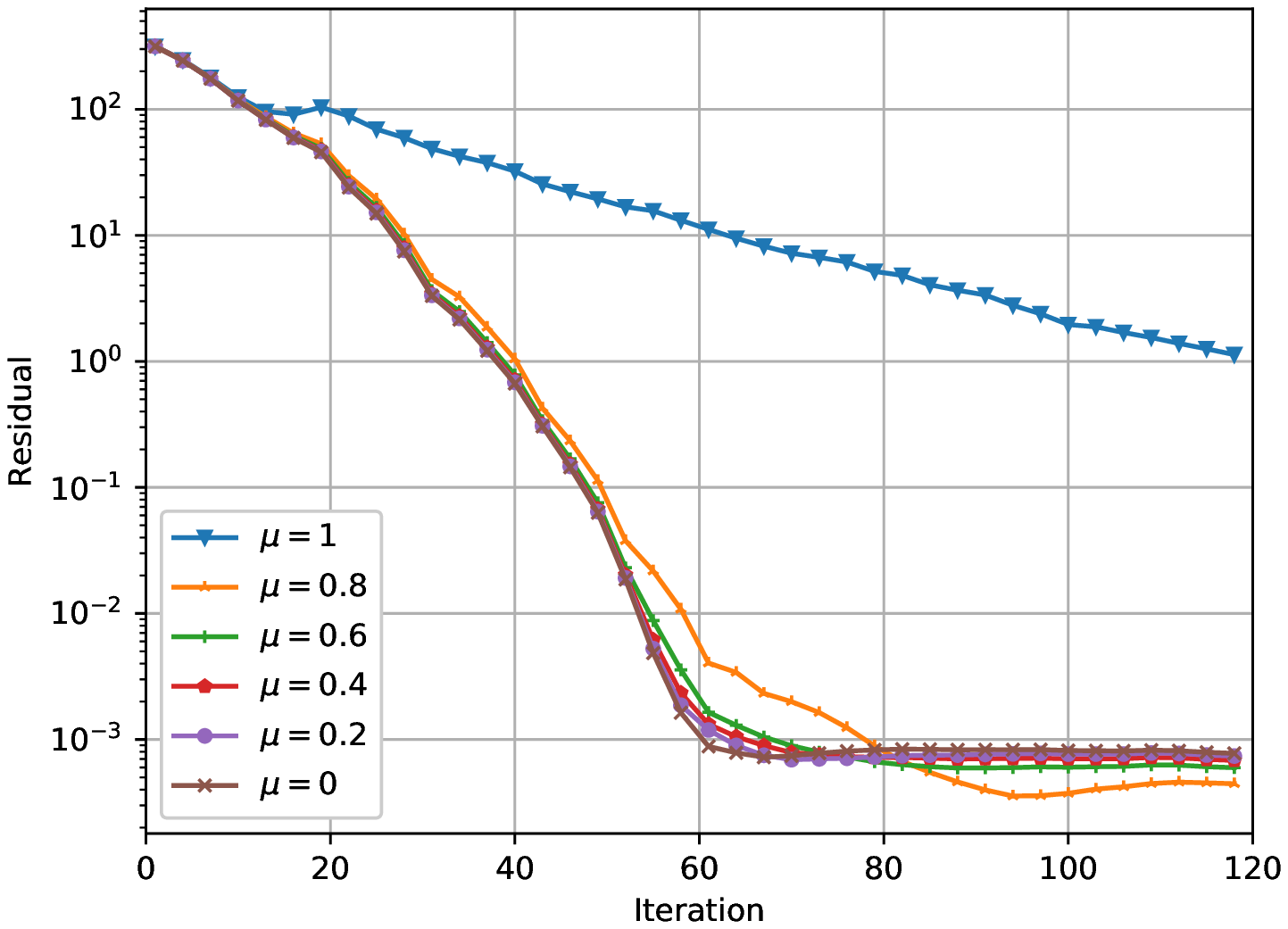}
\end{minipage}%
}%
\subfigure[Test 1 (Train step: 200)]{
\begin{minipage}[t]{0.32\linewidth}
\centering
\includegraphics[width=\linewidth]{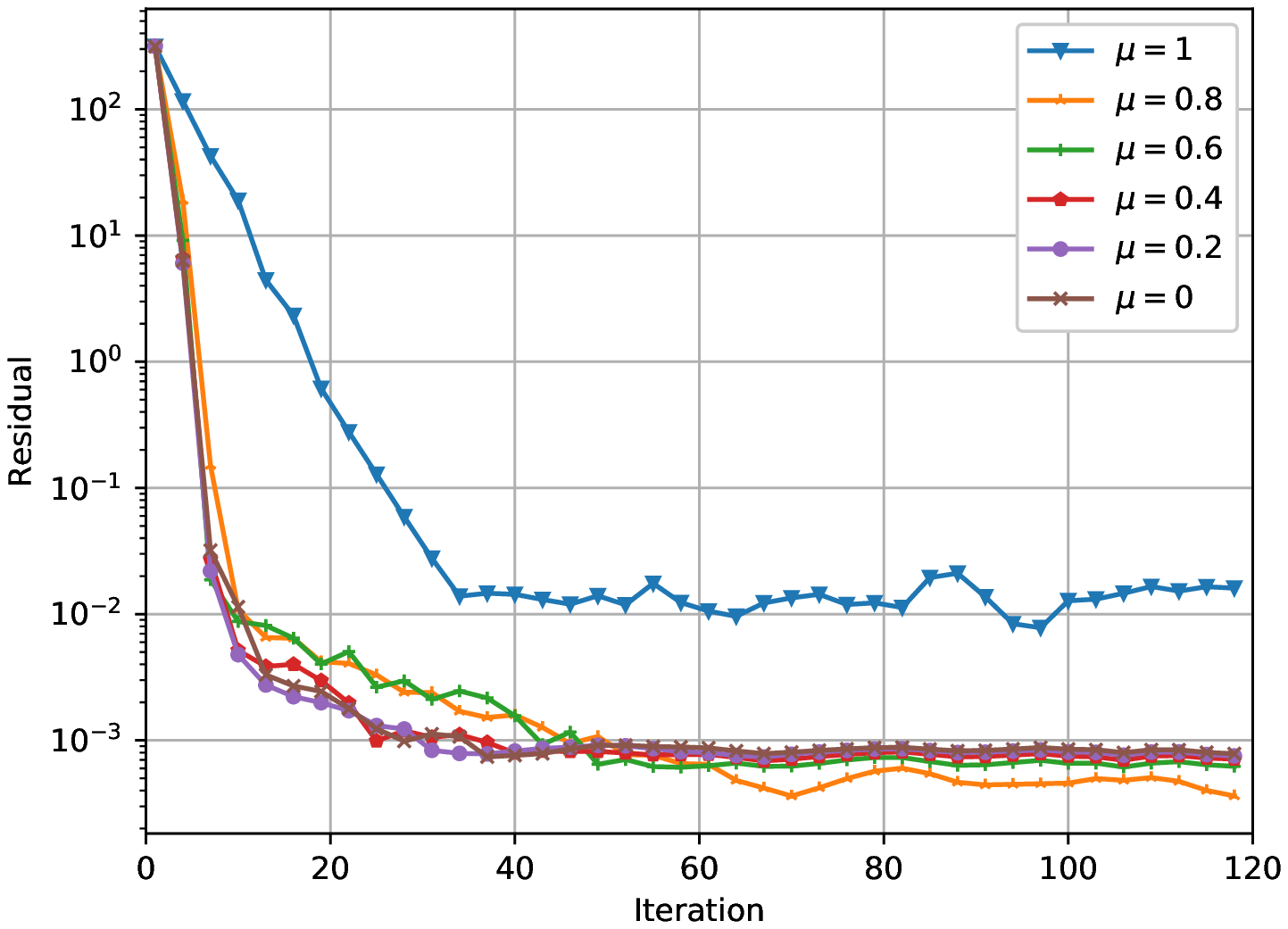}
\end{minipage}%
}%
\subfigure[Test 2 (Train step: 10)]{
\begin{minipage}[t]{0.32\linewidth}
\centering
\includegraphics[width=\linewidth]{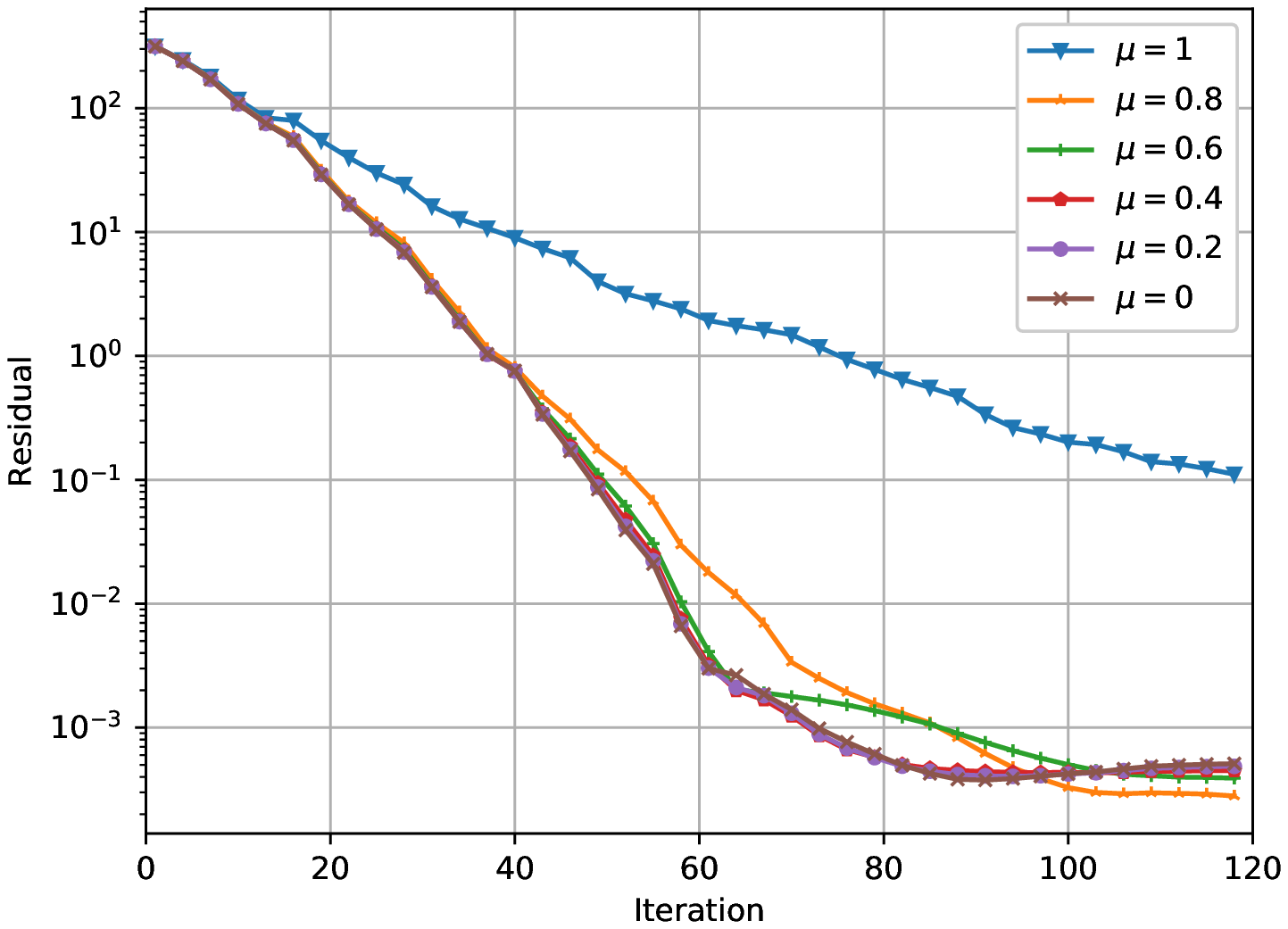}
\end{minipage}
}%

\subfigure[Test 2 (Train step: 200)]{
\begin{minipage}[t]{0.32\linewidth}
\centering
\includegraphics[width=\linewidth]{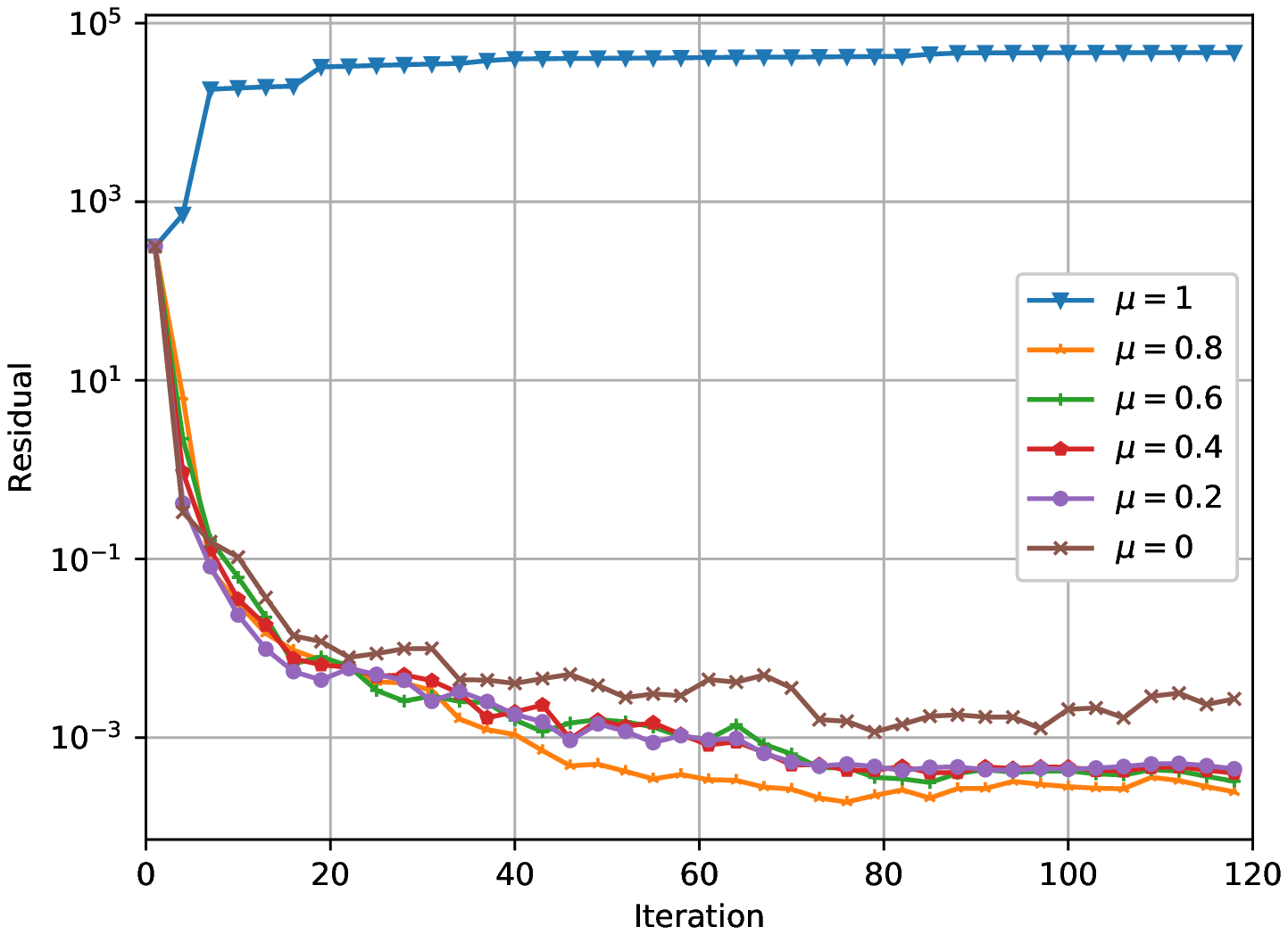}
\end{minipage}
}%
\subfigure[Test 3 (Train step: 10)]{
\begin{minipage}[t]{0.32\linewidth}
\centering
\includegraphics[width=\linewidth]{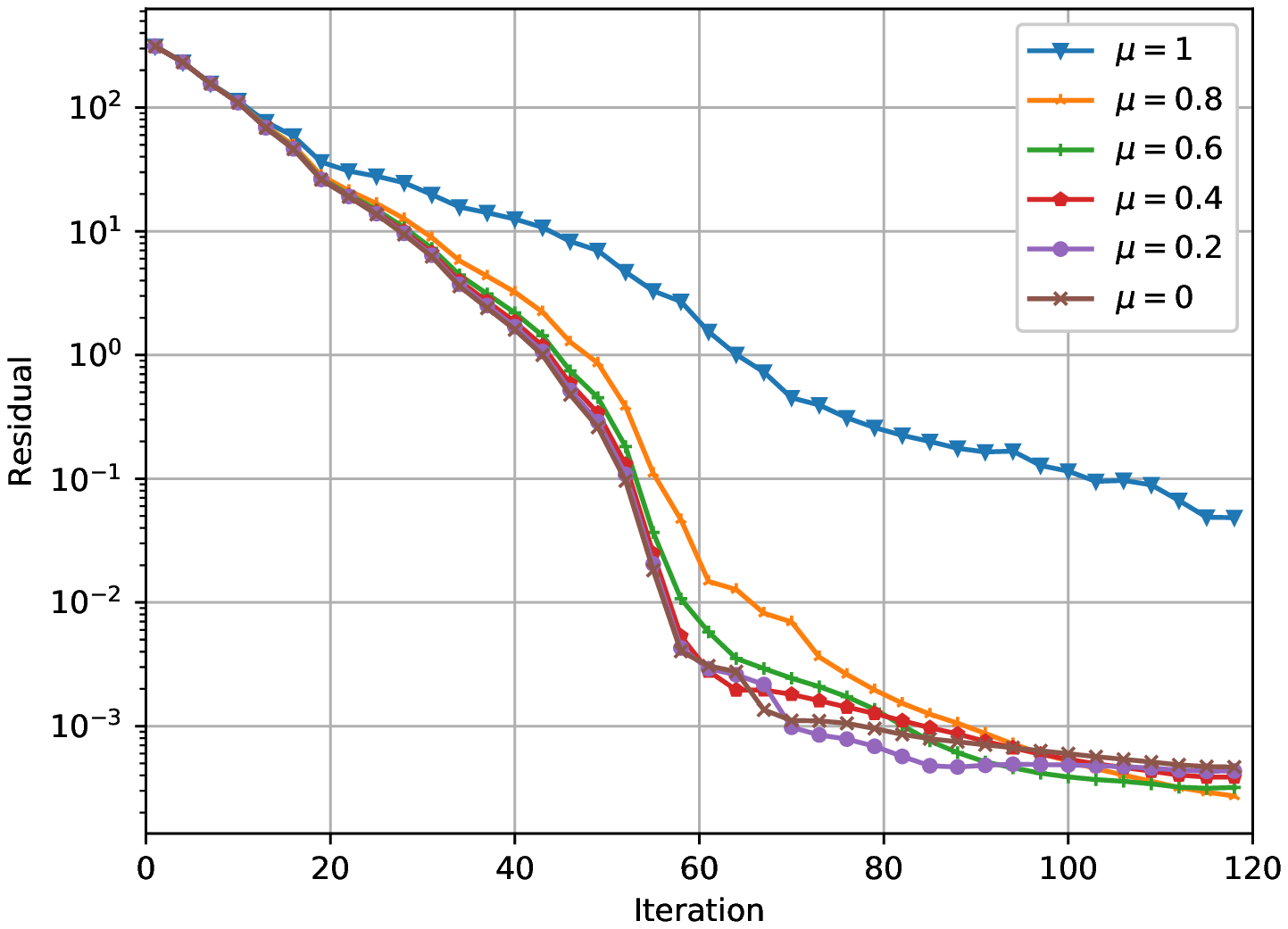}
\end{minipage}%
}%
\subfigure[Test 3 (Train step: 200)]{
\begin{minipage}[t]{0.32\linewidth}
\centering
\includegraphics[width=\linewidth]{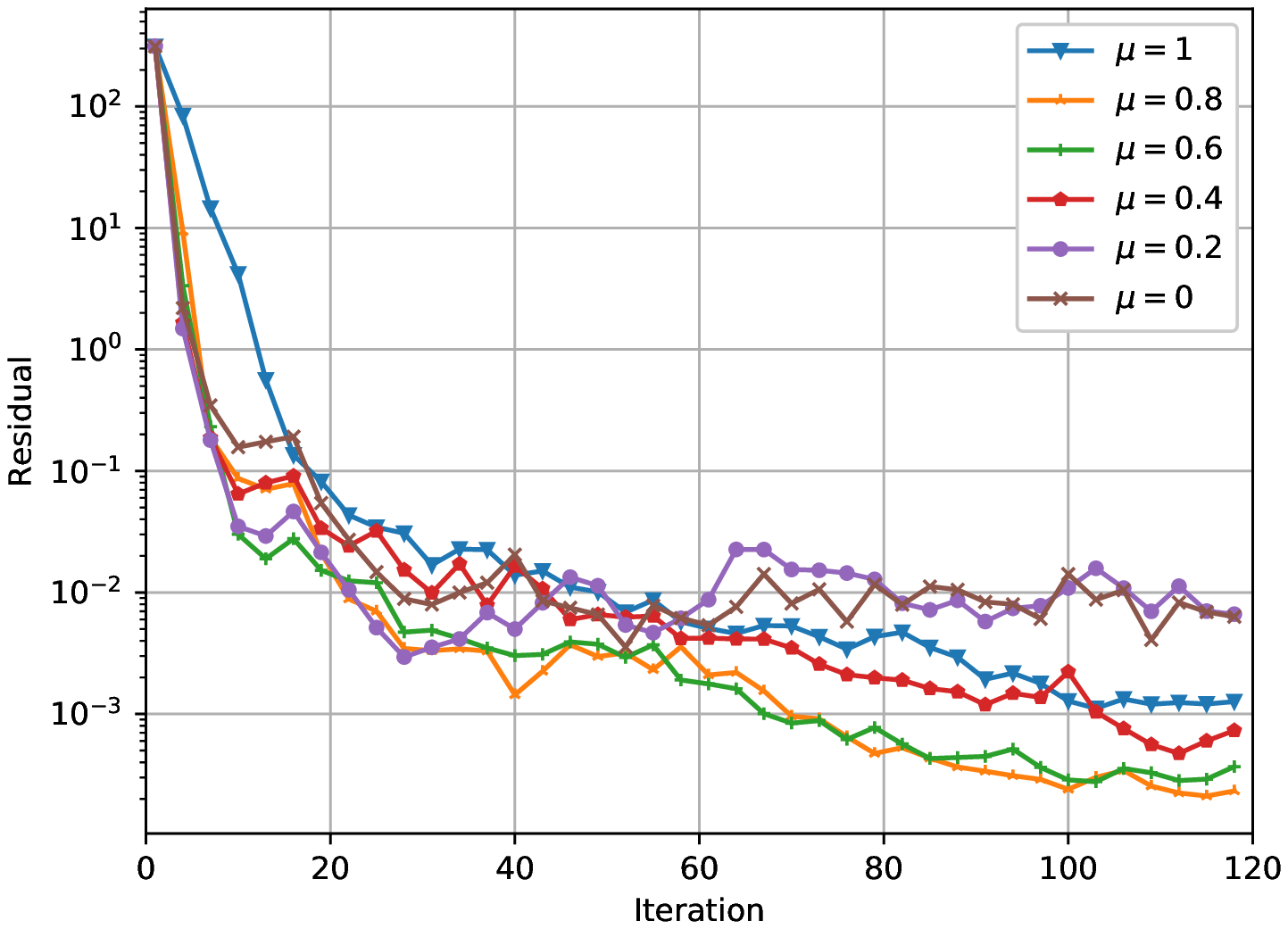}
\end{minipage}%
}%
\caption{Error (HJB residual) vs policy iteration for various $\mu$. The trajectory number is 5.}
\label{fig:lqr_train}
\end{figure*}

 Fig. \ref{fig:lqr_train} shows   the convergence  of the policy iteration at  different $\mu$ values.
 We see that again for all three tests,  $\mu=1$ which corresponds to
 the policy iteration with only the value function
 gives the slowest decay of error among all $\mu$ tests.
 For other values of $\mu<1$,  the performances of reducing the error 
 are basically similar and all outperform the case of $\mu=1$.

 In summary, for the toy model of  linear-quadratic problem, 
 we have conducted many numerical tests
 to show the advantage of our formulation
 of using the value-gradient data in training the 
 value function:
 it improves the convergence of the policy iteration
 and shows much better robustness
 for a limited amount of data and a
 limited number of training steps.

\subsection{Cart-pole balancing}
Cart-pole balancing task is a 4-dim nonlinear case \cite{Barto1983}.
The physical model of this task includes a car, a pole and a ball. The ball is connected to one end of the pole and the other end of the pole is fixed to the car. The pole can rotate around the end fixed to the car, while the car is put on a flat surface, being able to  move left or right. The aim of this task is to balance the pole in the upright vertical direction.

The state variable has four dimensions: the angular velocity of the ball, denoted by $\omega$; the included angle of the pole and the vertical direction, denoted by $\psi\in[-\pi,\pi]$; the velocity of the car, denoted by $v$; the position of the car, denoted by $z$. The control of this problem is the force applied to the car, denoted by $F$.

The control problem is to let $\psi$ be as small as possible. To eliminate the translation
invariant in the horizontal position, we also want  $z$ to be small. 
So we aim to  minimize $-\cos(\psi)$ and $|z|^2$ with the following cost
function 
$$J(u)=\int_0^\infty  e^{-\rho t}\left(-\cos (\psi(t))+\eta |z(t)|^2\right) \text{d}t$$
with $\rho=5$ and $\eta=0.2$ and 
subject to the dynamical system
$$\left\{
\begin{aligned}
 &\dot{\omega} =\frac{g \sin \psi+\frac{\left(\mu_{c} \operatorname{sgn}(v)-F-m l \omega^{2} \sin \psi\right) \cos \psi}{m+m_{c}}-\frac{\mu_{p} \omega}{m l}}{l\left(\frac{4}{3}-\frac{m}{m+m_{c}} \cos ^{2} \psi\right)} \\ &\dot{\psi} =\omega  \\&\dot{v} =\frac{F+m l\left(\omega^{2} \sin \psi-\dot{\omega} \cos \psi\right)-\mu_{c} \operatorname{sgn}(v)}{m+m_{c}} \\  &\dot{z}=v\\
 &\omega(0)=\omega_0,\psi(0)=\psi_0,v(0)=v_0,z(0)=z_0
\end{aligned}\right.
$$where $m$ is the mass of the ball, $m_c$ is the mass of the car, $l$ is the length of the pole, $g$ is the gravitational constant. A constraint is imposed to the control: $|F|\leq F_{max}$, $F_{max}\geq 0$ is the largest control we can have. These hyper-parameters are set to be
\begin{equation*}m=0.1,\ l=0.5,\ m_c=1,\ \mu_c=5\times 10^{-4},\ \mu_p=2\ \times 10^{-6},\ F_{max}=10. \end{equation*}

The state variable is $x=(\omega,\psi,v,z)\in \Real^4$ with $\psi\in [-\pi, \pi)$. 
The value function is approximated by neural network with radial basis function ($n=50$ modes). Totally, there are $(2d+1)n=450$ parameters to learn.
We compute the value function on the domain  $\Omega=[-2\pi,2\pi]\times[-\pi,\pi]\times[-0.5,0.5]\times[-2.4,2.4]$. So the initial values of the characteristics $X_0^{(n)}$ are uniformly
sampled from $\Omega$. But the characteristics are computed in the whole space with sufficiently long time
until $e^{-\rho t}\Phi(X(t))$ and  $e^{-\rho t}\lambda(X(t))$ are both sufficiently small. Error of the numerical solution  $\widehat{\Phi}_{\theta}$  is  measured by the HJE residual calculated on  
 $N_p=10000$ points uniformly sampled  from $\Omega$:
\begin{equation}
\label{eq:DQ}
\begin{split}
error=\frac{1}{N_p}\sum^{N_p}_{j=1}
\Big\|
\rho \hat{\Phi}_{\theta}\left(x^{(j)}\right)
-g&\left(x^{(j)}, a^*\left(x^{(j)}\right)\right) \cdot \nabla \hat{\Phi}_{\theta}\left(x^{(j)}\right) -l\left(x^{(j)}, a^*\left(x^{(j)}\right)\right)\Big\|,
\end{split}
\end{equation}where $x^{(j)}$ is the $j$-th data point.
 
 To better evaluate the performance, we introdue the "successful roll-up": in a 20 second simulation ($T=20$), if
\begin{itemize}
\item $|\psi(t)|<\pi/4$ lasts for at least 10 seconds;
\item  $|z(t)|<10$ for all $t\in[0,T]$.
\end{itemize}
Then we call this run a "successful roll-up".

  The initial condition for measuring the   successful roll-up numbers
   are $(\omega(0),\psi(0),\ \ $ $v(0)=0,z(0)=0)$ with  100 pairs of $(\omega(0),\psi(0))$ from the $10\times 10$ mesh grid of $[-2\pi,2\pi)\times[-\pi,\pi).$

We conduct the same two experiments as in the Linear-quadratic problem for this case which test the performance under insufficient data or incomplete training.

 {\bf Experiment 1}.
In Experiment 1, we study how insufficient amount of  
characteristics data will affect the performance.
Specifically, we test the performance of trajectory numbers of 2, 5 and 10 while the training for the supervised learning to minimize the loss $L(\theta)$ takes a fixed number of 50 ADAM steps. Fewer trajectories mean less amount of labelled data for the method of characteristics.

\begin{table}[htbp!]
\scriptsize
\centering
\begin{tabular}{m{1.5cm}<{\centering}|m{1.22cm}<{\centering}|m{1.1cm}<{\centering}m{1.1cm}<{\centering}m{1.1cm}<{\centering}}
\hline\hline\multirow{2}{1.2cm}{}&  \multirow{2}{1.2cm}{\centering $\mu$} & \multicolumn{3}{c}{Number of trajectories} \\
\cline{3-5}
& & 2 & 5 & 10 \\
\hline 
\multirow{6}{1.5cm}{\centering Residual}&$1.0$ & \textit{2.088} &  \textit{0.844} & \textit{0.934} \\
&$0.8$ & 0.696 & 0.281 & 0.100 \\
&$0.6$ & 0.450 & 0.169 & 0.117 \\
&$0.4$ & 0.441 & 0.147 & 0.091 \\
&$0.2$ & 0.181 & \textbf{0.113} & \textbf{0.082} \\
&$0.0$ & \textbf{0.166} & 0.124 & 0.094 \\
\hline\hline
\multirow{6}{1.5cm}{\centering Successful roll-up}&$1.0$ & 14.85 & \textit{19.25} & \textit{12.30} \\
&$0.8$ & 10.65 & 25.10 & \textbf{60.8} \\
&$0.6$ & \textbf{27.10} & 25.15 & 38.65 \\
&$0.4$ & 11.05 & 39.25 & 44.20 \\
&$0.2$ & \textit{9.30} & 40.95 & 50.45 \\
&$0.0$ & 20.95 & \textbf{44.95} & 55.00 \\
\hline\hline
\end{tabular}
    \caption{ The error (HJE residual) and the number of 
    successful roll-ups for different $\mu$ when  the  trajectory number $N$ 
    changes in the cart-pole balancing task. The train step is 100.}
    \label{tab:car_tranum_residual}
\end{table}

 Table \ref{tab:car_tranum_residual} shows the results when $\mu$ varies
for each test. For each given $N$,
the collection of  $N$ initial states 
are the same at different $\mu$ for consistent comparison.
The average residual error of the last 20 iterations is reported in the table.   For each setting,  the best residual is highlighted  in bold symbols and the worst residual is emphasised in  italics. From the table, we can see that $\mu=1$ performs the worst in all cases. In fact, a huge improvement can be observed in the residual and successful roll-up number when the gradient information is used. Also, this table confirms that with the number of characteristics  increasing, the final accuracy of the numerical value functions always 
gets better and better since more labelled data are provided.

\begin{figure*}[htbp!]
\centering
\subfigure{
\begin{minipage}[t]{0.48\linewidth}
\centering
\includegraphics[width=0.95\textwidth]{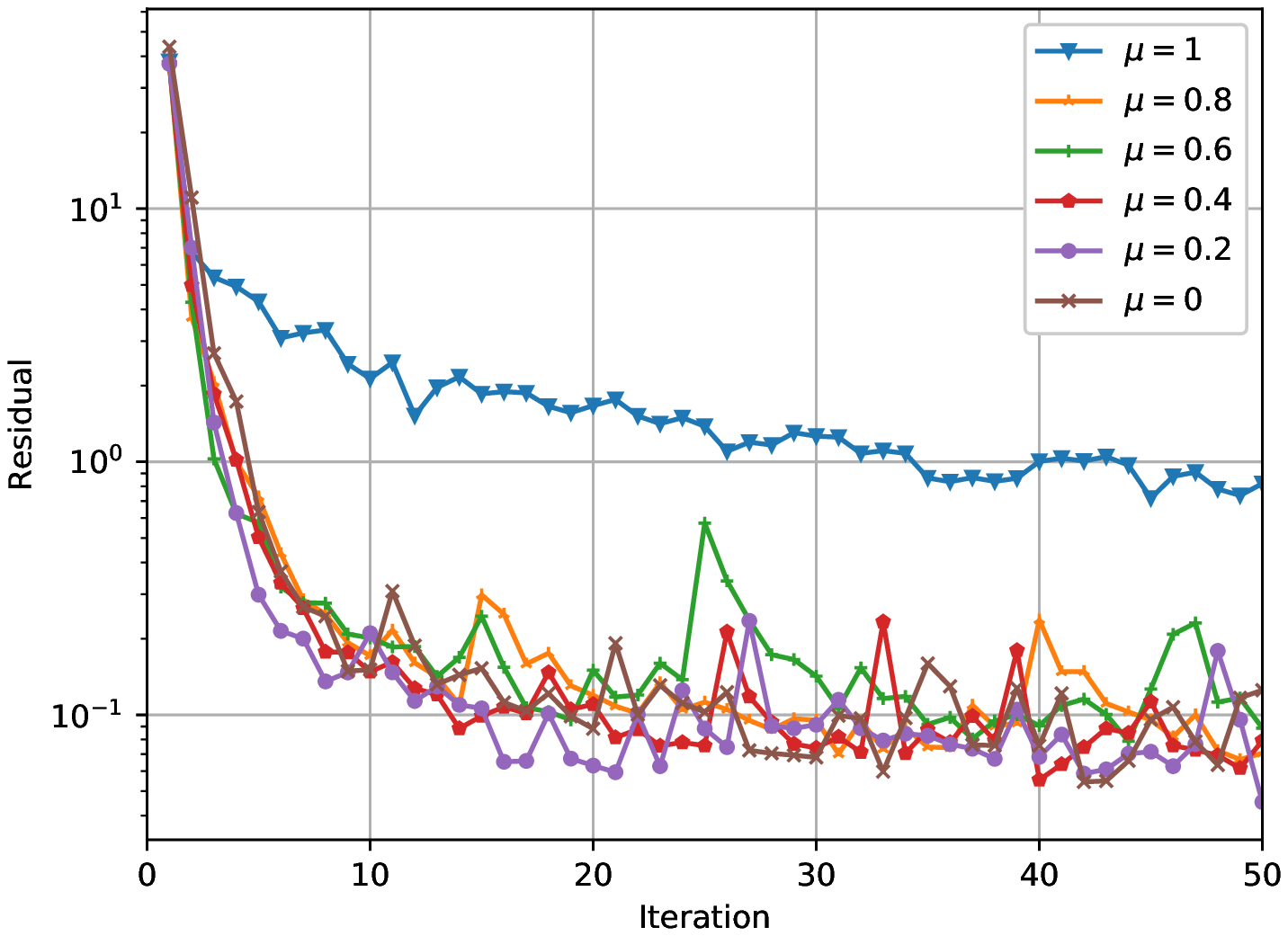}
\end{minipage}
}%
\subfigure{
\begin{minipage}[t]{0.48\linewidth}
\includegraphics[width=0.95\textwidth]{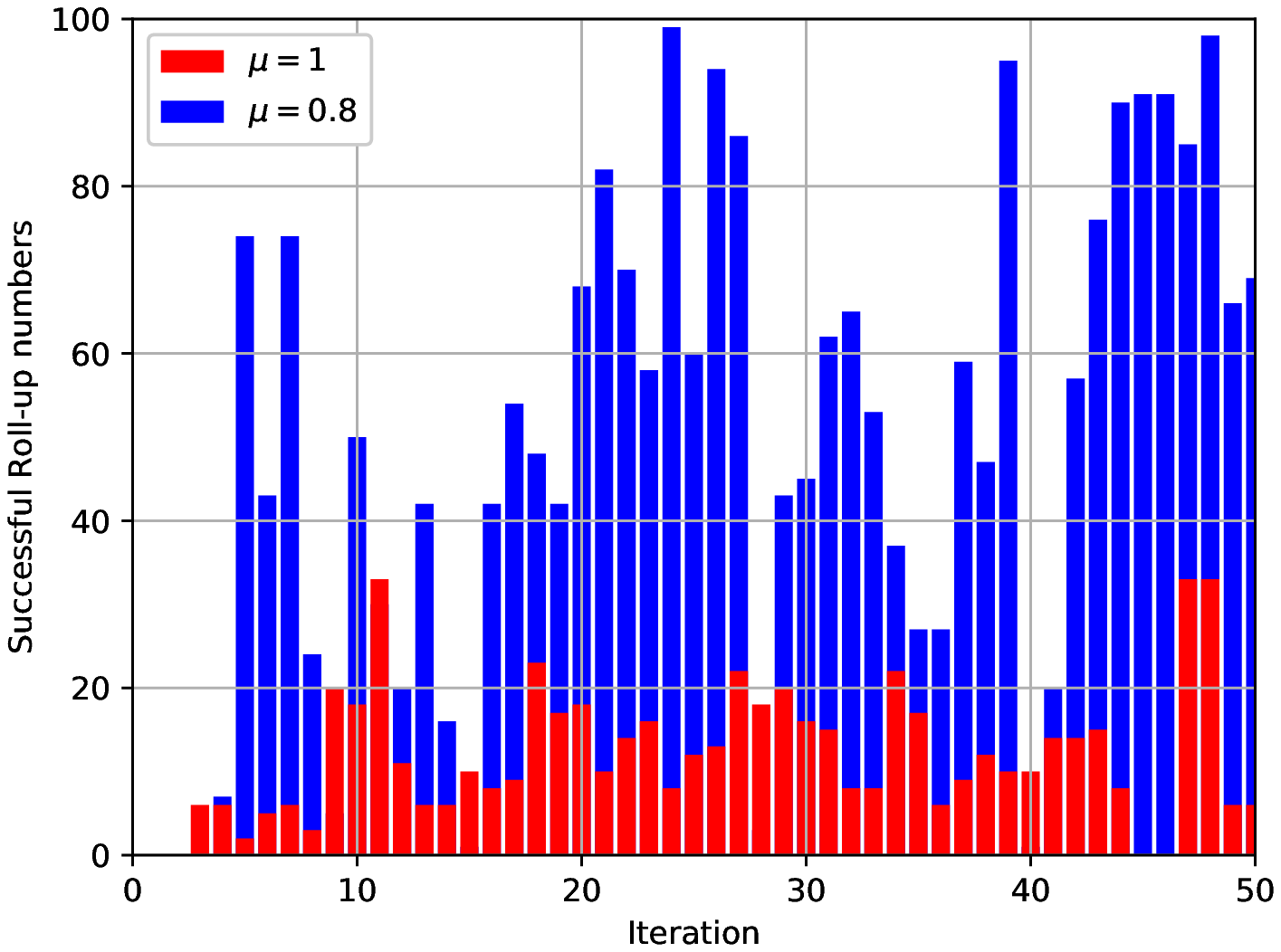}
\end{minipage}%
}%
\caption{Residual and the successful roll-ups with trajectory number of 10 and train step of 100 in 
the cart-pole task.}\label{fig:4dim_tra}
\end{figure*}

To investigate the effect of $\mu$ on  the decay of the error,
we plot the residual error during the policy iteration in Fig. \ref{fig:4dim_tra}. 
 This figure clearly  demonstrates that $\mu=1$ has the slowest 
 convergence  among all $\mu$ we tested,
and we can find that  adding even a small portion of the loss for the value-gradient,
i.e., $\mu<1$, can improve the convergence. Also, we plot the successful roll-up number of $\mu=1$ compared with $\mu=0.8$ at each iteration. It can be seen that value-gradient significantly improves the performance.

 {\bf Experiment 2}.
The purpose of  Experiment 2 is to test the performance of the methods when the training process is not sufficiently long. 
In this experiment,   the train steps   50, 100, 150 and 200 are tested.
A small training step means less accuracy in fitting the value function. 
 The trajectory number is now fixed as 10.

\begin{table}[htbp!]
\scriptsize
\centering
\begin{tabular}{m{1.5cm}<{\centering}|
m{1.22cm}<{\centering}|m{1.1cm}<{\centering}m{1.1cm}<{\centering}m{1.1cm}<{\centering}m{1.1cm}<{\centering}}
\hline\hline \multirow{2}{1.2cm}{}&  \multirow{2}{1.2cm}{\centering $\mu$} & \multicolumn{4}{c}{Train step} \\
\cline{3-6}
& & 50 & 100 & 150 & 200 \\
\hline \multirow{6}{1.2cm}{\centering Residual}&$1.0$ &\textit{1.355} & \textit{0.934} & \textit{0.500} & \textit{0.471} \\
&$0.8$ & 0.260 & 0.100 & 0.151 & 0.155 \\
&$0.6$ & 0.175 & 0.117 & \textbf{0.092} & 0.097 \\
&$0.4$ & \textbf{0.096} & 0.091 & 0.105 & 0.103 \\
&$0.2$ & 0.106 & \textbf{0.082} & 0.094 & \textbf{0.080} \\
&$0.0$ & 0.130 & 0.094 & 0.070 & 0.083 \\
\hline
\hline
\multirow{6}{1.5cm}{\centering Successful roll-up }&$1.0$ & \textit{13.75} & \textit{12.30} & 28.55 & \textit{30.00} \\
&$0.8$ & 58.25 & \textbf{60.80} & 41.45 & 35.80 \\
&$0.6$ & 28.05 & 38.65 & \textit{16.50} & 35.55 \\
&$0.4$ & 56.40 & 44.20 & 36.50 & 35.75 \\
&$0.2$ & \textbf{61.65} & 50.45 & 35.30 & 47.00 \\
&$0.0$ & 21.85 & 55.00 & \textbf{49.30} & \textbf{54.10} \\
\hline\hline
\end{tabular}
\caption{The error (HJE residual) and the number of 
    successful roll-ups  for different $\mu$
    when training steps change in the cart-pole balancing task. The trajectory number is 10.}
    \label{tab:car_trainstep}
\end{table}

As shown in 
Table \ref{tab:car_trainstep}, the accuracy gets quite remarkable improvements
as long as the value-gradient is included  in the formulation.
The successful roll-ups  also show a better performance for $\mu<1$, 
particularly when the number of trajectories increases.

\begin{figure*}[!htbp]
\centering
\subfigure{
\begin{minipage}[t]{0.48\linewidth}
\centering
\includegraphics[width=0.95\textwidth]{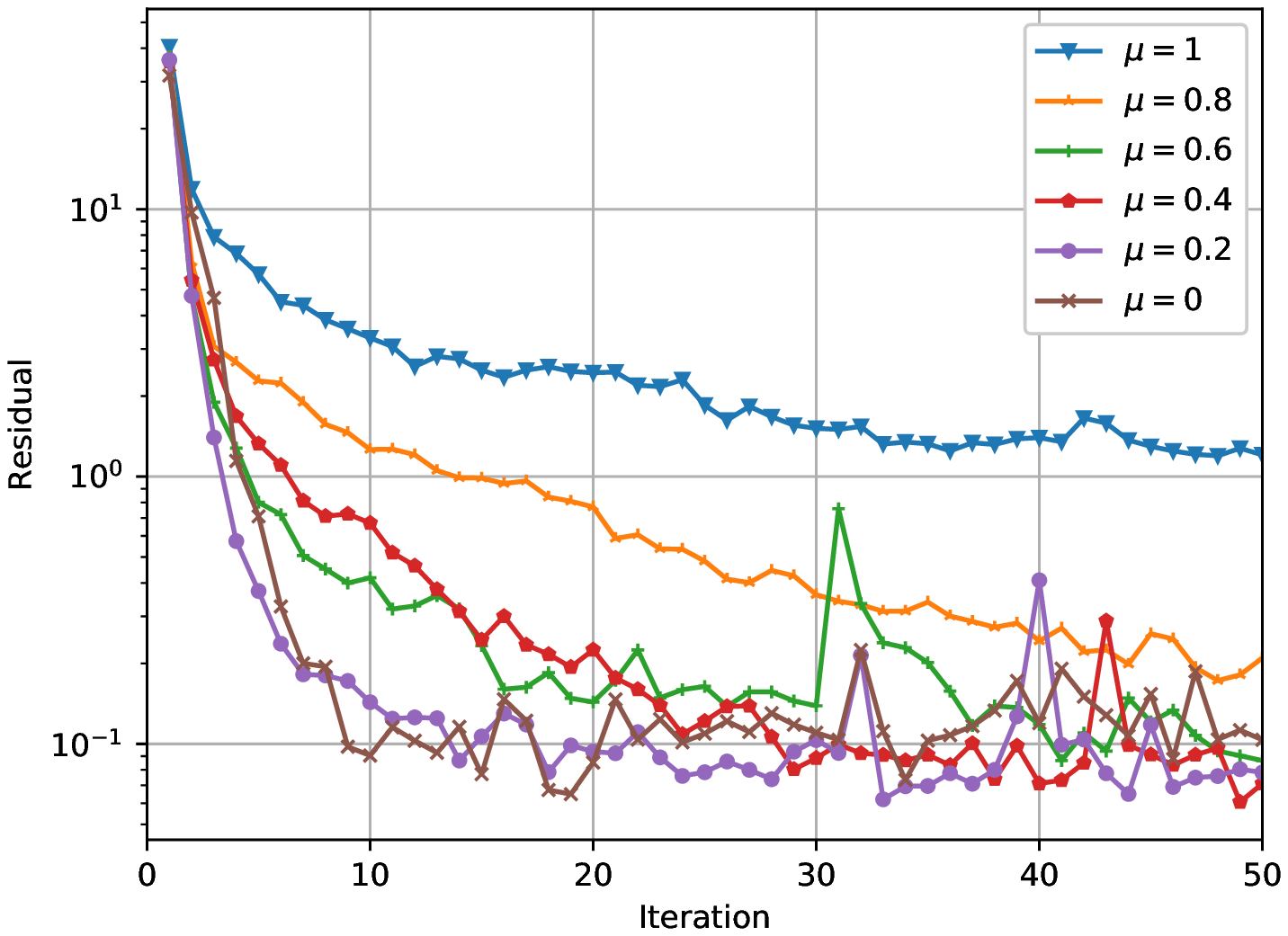}
\end{minipage}%
}%
 \subfigure{
 \begin{minipage}[t]{0.48\linewidth}
 \centering
 \includegraphics[width=0.95\textwidth]{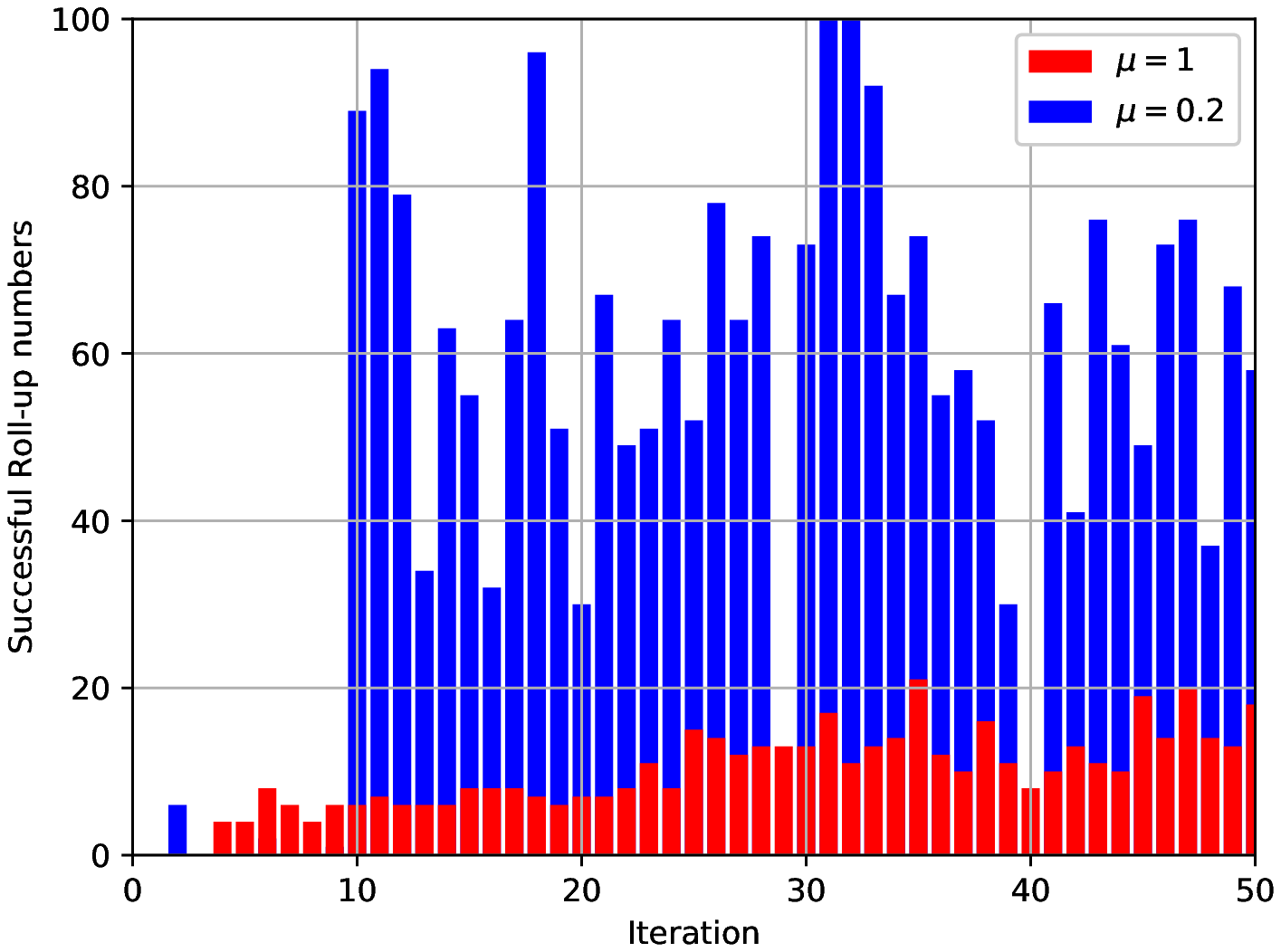}
 \end{minipage}%
 }%
\caption{Residual and the successful roll-ups with trajectory number of 10 and train step of 50 in 
the cart-pole task.}\label{fig:4dim_train}
\end{figure*}

 Fig. \ref{fig:4dim_train}
shows  the residual and successful roll-ups with respect to 
the policy iteration for different $\mu$ values. 
As expected,  choosing $\mu<1$ gets these results considerably improved.

\subsection{Advertising process}
This example is a 3-dim nonlinear case from \cite{Vienna1988,Feichtinger1994}. The three dimensions of the states are the advertising stimulus level $A$, the adaptation level $\overline{A}$ and sales $S$. We are aiming at finding the optimal advertising effort $u$ that maximize the cost function
$$
J(u)=\int_{0}^{\infty}e^{-\rho t}(\pi S(t)-u(t)) \d t,
$$
subject to the dynamic system
$$
\left\{
\begin{aligned}
&\dot{A}=u-\delta A, A(0)=A_{0}\\
&\dot{\overline{A}}=\zeta(A-\overline{A}), \overline{A}(0)=\overline{A}_{0} \\
&\dot{S}=v \ln (A+1)-\alpha S+\overline{w}\max\{0,A-\overline{A}\}, S(0)=S_{0}
\end{aligned}
\right.
$$
where $\delta$ is a constant proportional depreciation rate, $\zeta>0$ represents the relative weight of more rescent levels of advertising capital, $\alpha$ denotes the proportion of customers switching to other brands per unit time, $\pi$ is the gross profit per unit sold and $\overline{w}$ and $v$ are constants. The control $u$ has upper bound and lower bound $0\leq u\leq \overline{u}$. All the hyper-parameters are set to:
$$\overline{u}=2,\delta=0.5,\zeta=1,v=0.5,\alpha=0.1,\overline{w}=0.5,\pi=0.5$$
The value function is parametrized as a family of radial basis functions with 60 modes.
We also conduct two experiments on this problem as previous examples. The total number of policy iteration is 200.

\textbf{Experiment 1}.
In Experiment 1, trajectory numbers of 2, 5, 10 are tested. The training step is fixed to 50. Table \ref{tab:adv_tranum_residual} records the average HJB residual \eqref{eq:DQ} of the last 40 iterations of the 200 policy iterations. As is shown in the table, $\mu=1$ has the worst residual among all. Fig. \ref{fig:3dim_tra} demonstrates the residual with respect to the policy iteration number. $\mu\in[0,1)$ converges faster and performs better than $\mu=1$.

\begin{table}[htbp]
\scriptsize
\centering
\begin{tabular}{m{1.22cm}<{\centering}|m{1.1cm}<{\centering}m{1.1cm}<{\centering}m{1.1cm}<{\centering}}
\hline\hline  \multirow{2}{1.2cm}{\centering $\mu$} & \multicolumn{3}{c}{Number of trajectories} \\
\cline{2-4}
 & 2 & 5 & 10 \\
\hline
$1.0$ & \textit{0.0627} & \textit{0.0429} & \textit{0.0286}\\
$0.8$ & \textbf{0.0373} & 0.0265 &0.0180\\
$0.6$ & 0.0506 & 0.0251 &0.0180\\
$0.4$ & 0.0511 & 0.0205 &0.0193\\
$0.2$ & 0.0447 & 0.0237 &0.0106\\
$0.0$ & 0.0417 & \textbf{0.0107} &\textbf{0.0080}\\
\hline\hline
\end{tabular}
    \caption{The HJB residual for different $\mu$ in advertise process task Experiment 2. 
    Smaller residuals are better results. 
    The best value for each trajectory number is highlighted in bold symbols and the worst
    is marked in italics.
    The training step is 50. }
    \label{tab:adv_tranum_residual}
\end{table}

\begin{figure*}[htbp]
\centering

\subfigure[Trajectory Number: 2]{
\begin{minipage}[t]{0.48\linewidth}
\centering
\includegraphics[width=.95\textwidth]{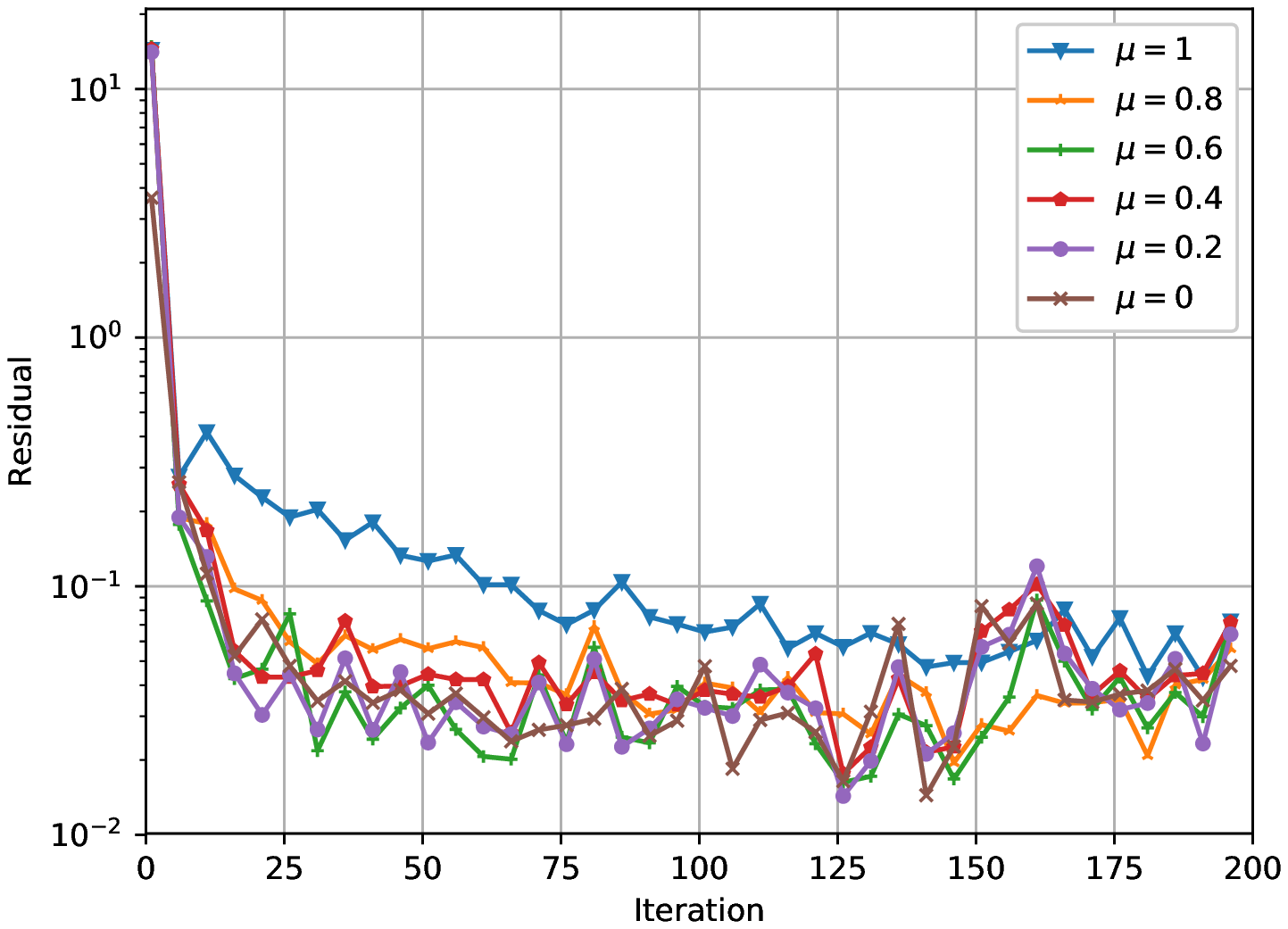}
\end{minipage}%
}%
\subfigure[Trajectory Number: 10]{
\begin{minipage}[t]{0.48\linewidth}
\centering
\includegraphics[width=0.95\textwidth]{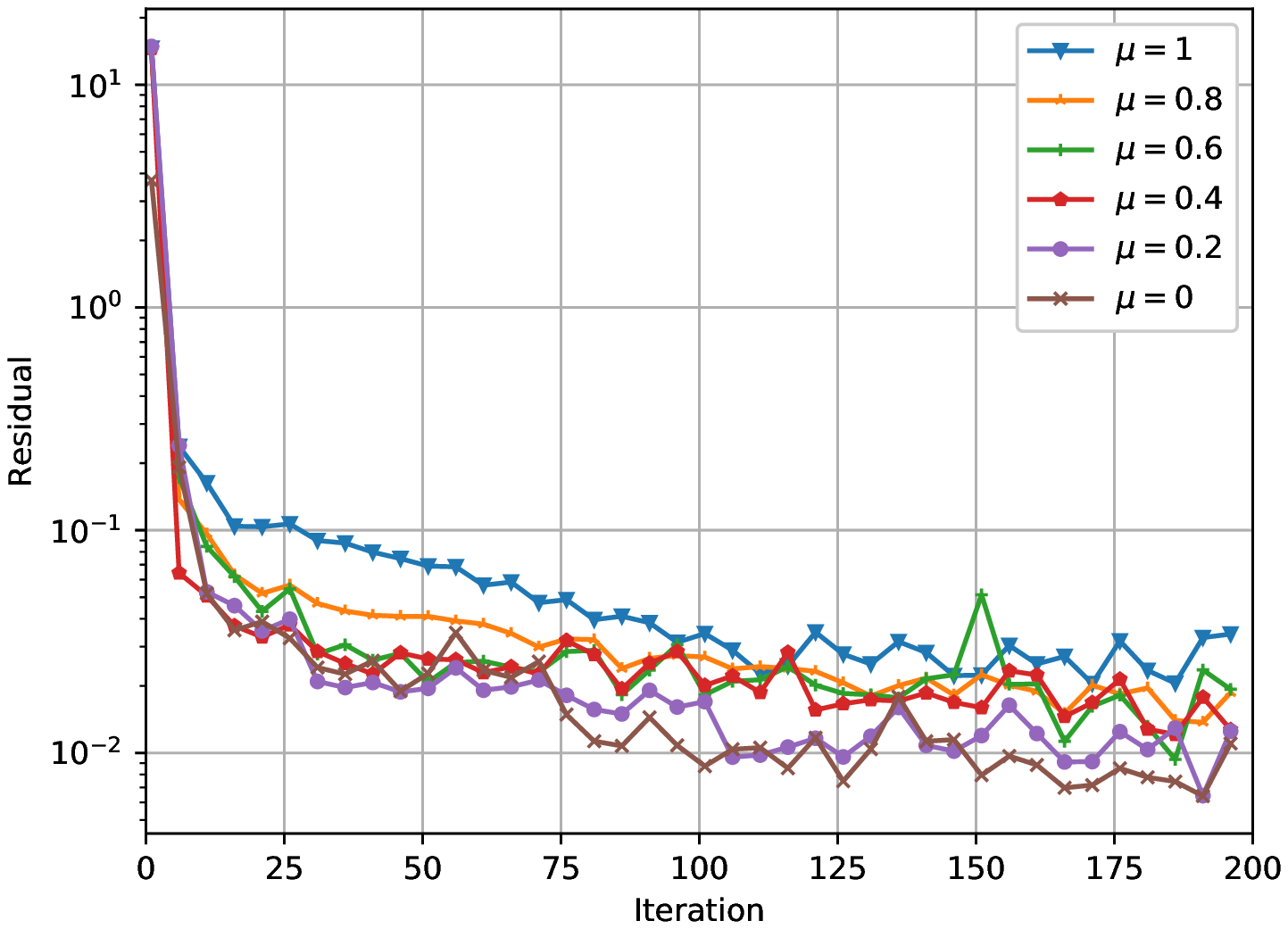}
\end{minipage}
}%

\caption{Residual during the 200 policy iteration 
for various $\mu$ in  the advertise process task. The trajectory numbers are 2 and 10 and the training step is 50. 
}\label{fig:3dim_tra}
\end{figure*}

 {\bf Experiment 2}.
In Experiment 2, we test the training step of 25, 50, 75, 100.  Table \ref{tab:adv_trainstep_residual} records the residual error and Fig. \ref{fig:3dim_train} demonstrates the residual with respect to the policy iteration number. It can be concluded that using a mixture of value and value gradient works better than using value only.  

\begin{table}[htbp]
\scriptsize
\centering
\begin{tabular}{m{1.22cm}<{\centering}|m{1.1cm}<{\centering}m{1.1cm}<{\centering}m{1.1cm}<{\centering}m{1.1cm}<{\centering}}
\hline\hline  \multirow{2}{1.2cm}{\centering $\mu$} & \multicolumn{4}{c}{Number of train steps} \\
\cline{2-5}
 & 25 & 50 & 75 & 100 \\
\hline
$1.0$ & \textit{0.04098} & \textit{0.0429} & \textit{0.03651} & \textit{0.03772}\\
$0.8$ & 0.03134 & 0.0265 & 0.02986 & 0.02800\\
$0.6$ & 0.02844 & 0.0251 & 0.02578 & 0.02731\\
$0.4$ & 0.02149 & 0.0205 & 0.02107 & 0.03110\\
$0.2$ & 0.02111 & 0.0237 & 0.01160 & 0.03220\\
$0.0$ & \textbf{0.01632} & \textbf{0.0107} & \textbf{0.01101} & \textbf{0.02522}\\
\hline\hline
\end{tabular}
    \caption{The HJB residual for different $\mu$ in advertise process task Experiment 2. 
    Smaller residuals are better results. 
    The best value for each train step selection is highlighted in bold symbols and the worst
    is marked in italics.
    The trajectory number is 5. }
    \label{tab:adv_trainstep_residual}
\end{table}

\begin{figure*}
\centering
\subfigure[Train step: 25]{
\begin{minipage}[t]{0.48\linewidth}
\centering
\includegraphics[width=0.95\textwidth]{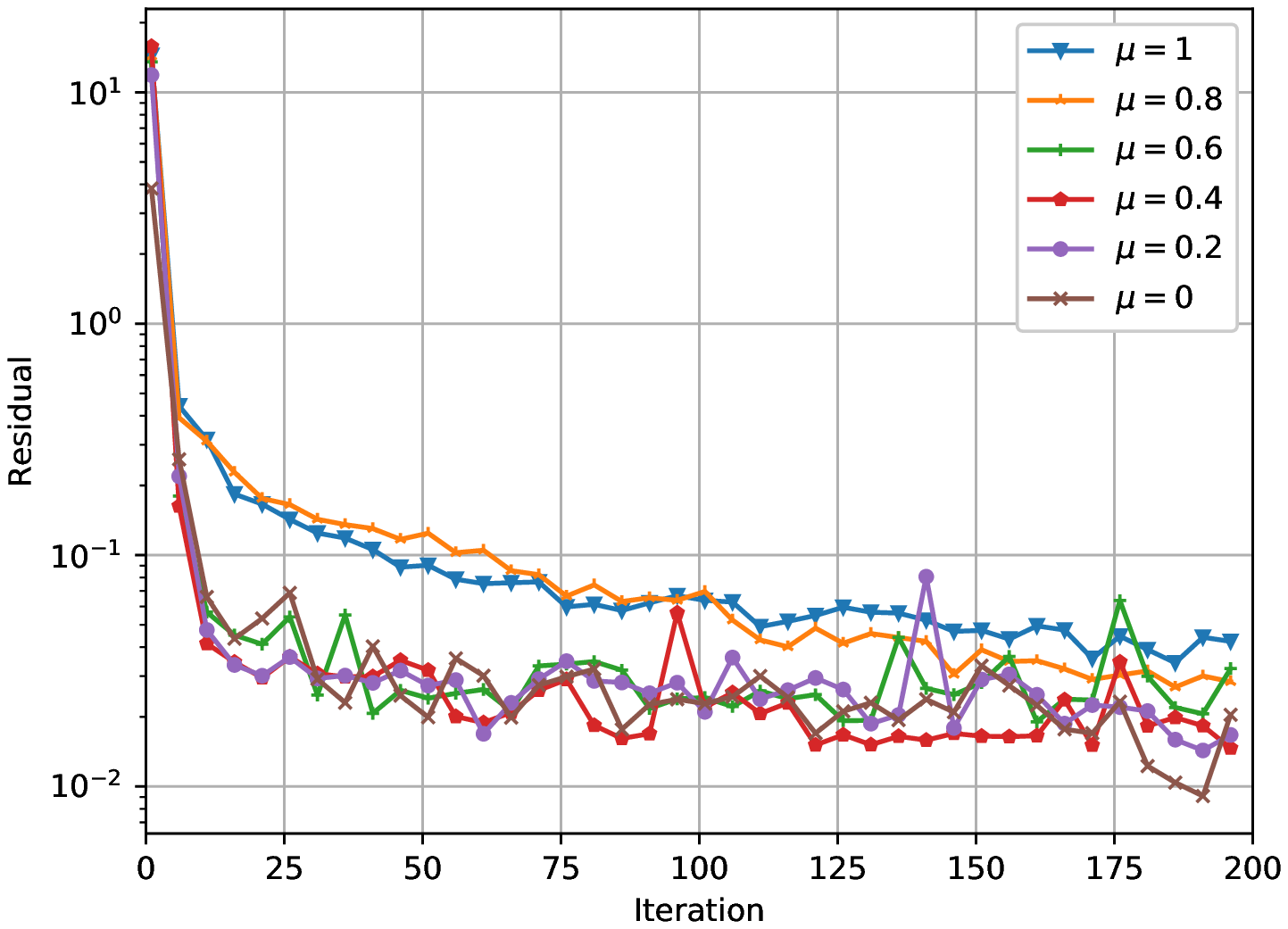}
\end{minipage}%
}%
\subfigure[Train step: 75]{
\begin{minipage}[t]{0.48\linewidth}
\centering
\includegraphics[width=0.95\textwidth]{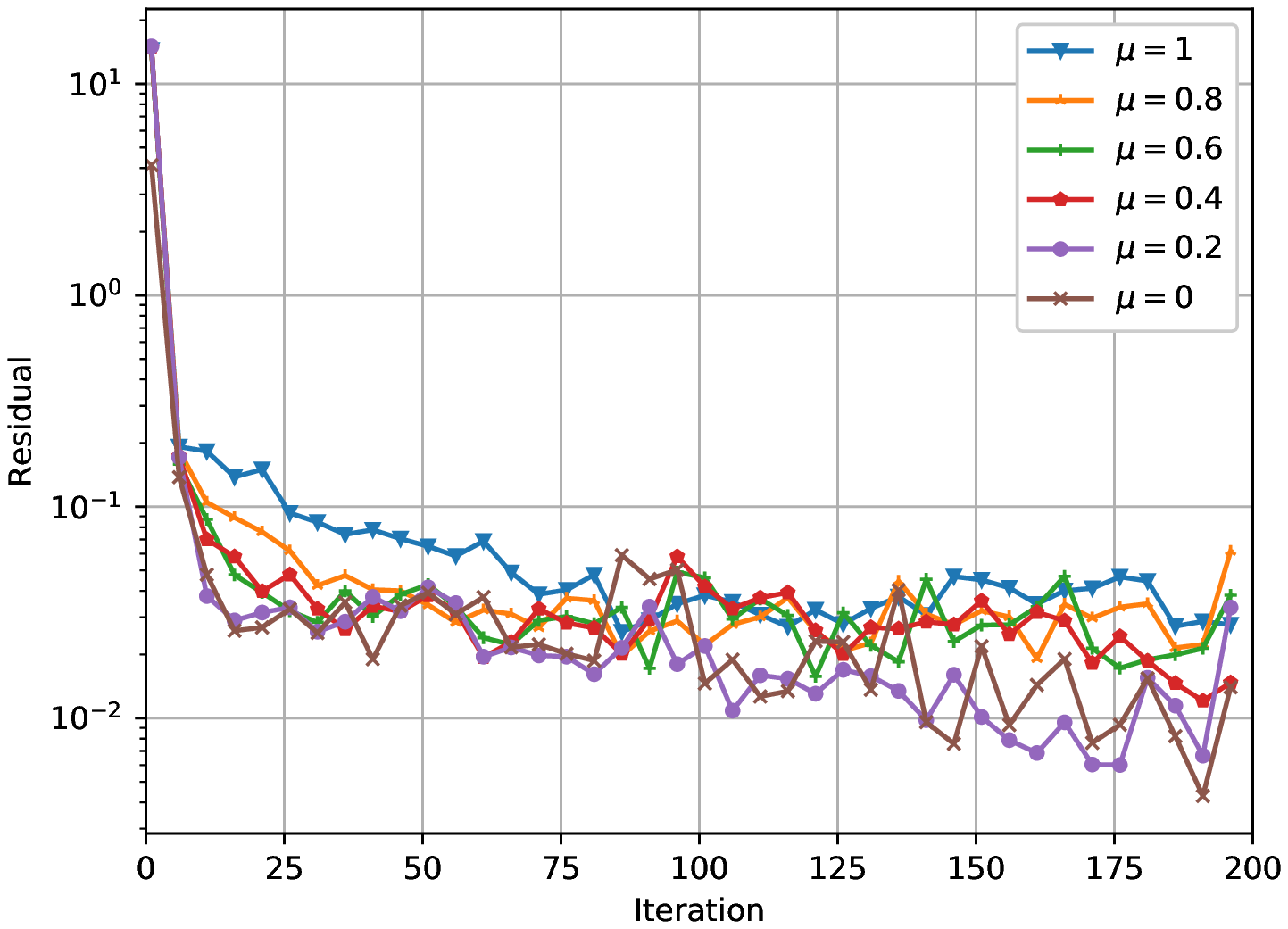}
\end{minipage}
}%

\caption{Residual during the policy iteration with training steps of 25 and 75 in advertise process task. The trajectory number is 5.
}
\label{fig:3dim_train}
\end{figure*}

To conclude the above experiment,
we have performed the numerical tests by changing the amount of characteristics data
and the training steps, which are two important factors in practical computation.
 By comparing the performance measured by the HJE residual as the error and the successful roll-ups
 as the robustness,
 we find that  these numerical results   consistently show 
the outperformance when using the characteristics data both from the value
and the value-gradient functions.
Although  the four tested values of $\mu=0.2,0.4,0.6,0.8$ between $0$ and $1$
always beat the traditional method at $\mu=1$, 
the optimal value $\mu$ actually varies on the specific settings
and the  difference among these four values for 
the performance is marginal.


\section{Conclusion}
\label{sec:con}
Based on the system of PDEs for the value-gradient functions we derived in this paper,
we develop a new policy iteration framework, called {\bf PI-lambda}, for the numerical solution of the value function for the optimal control problems.
We show the convergence property of this iterative scheme under \cref{ass:A1}.
The system  of PDEs for the value-gradient functions  $\lambda(x)$
is closed since it does not involve the value function $\Phi(x)$ at all,
so one could in principle  use neural networks only for $\lambda$.
This is  distinctive from many existing methods based on 
  value function (e.g. \cite{Han8505}).
The system for    $\lambda$
is also essentially decoupled and  shares the same characteristics ODE with the generalized HJE.
By simulating     characteristics curves in parallel for the state variable   by any classic ODE solver
(like Runge-Kutta method),   both the value $\Phi$ and the value-gradient functions $\lambda$ on  each characteristics curve can be computed.
Equipped with any state-of-the-art function representation technique 
and the large-scale minimization techniques from supervised learning,
 these labelled data can be generalized to the whole space to deal with high dimensional problems.
Policy iteration has the computational convenience to 
simulate the characteristics equations only forward in time, instead of solving any boundary-value problem
for {\it optimal} trajectories directly as in  \cite{Izzo01145,Kang2017,Kang2021QRnet}.
 Policy iteration is also convenient when Hamiltonian minimization 
 has no analytical expression.
The learning procedure of   supervised learning in our method is not new,
and it has been applied, for example in \cite{TailorIzzo2019,Izzo01145,Kang2019},
 to  combine the losses from  the policy data, the value function data and the value-gradient data altogether.  
 Our distinction  from these works  is to 
formulate   the co-state variable as the gradient function $\lambda(x)$ of the state,
not a function of the time $\lambda(t)$ in PMP.

The generalization to the finite horizon control problem on $[0,T]$ is straightfoward: to replace $\rho\lambda(x)$ by $-\partial_t \lambda(t,x)$
in equation \eqref{eq:lambda} and  add  the transversality condition $\lambda(T,x)=\nabla_x h(T,x)$ 
when there is a   terminal cost $h(T,x(T))$. The main algorithm in this paper based on the policy iteration,  {\bf PI-lambda}, is still applicable and our main theorem   (Theorem \ref{thm: main}) can be easily generalized.

Some practical computational issues which are not fully discussed here include 
the choice of the initial policy $a^{(0)}$, the number of trajectories $N$ and  their initial locations $\set{X_0}$.
For the initial policy,  it should be chosen conservatively to stabilize the dynamics. For the characteristics curves,  
$N$ may be changed from iteration to iteration,  and  adaptive sampling 
for the initial states is a good issue for further exploration \cite{Kang2019}. 
If a neural network is used, the network structure is also an important practical issue \cite{Kang2021QRnet}.

An obvious  question to address in future is how to formulate the equations of $\lambda$ for the stochastic optimal control so as to leverage the similar  benefit of our algorithm here for the deterministic  control problem. One may consider the splitting method in \cite{Bensoussan1992}.
 
\section*{Acknowledgment}
We thank Dr Bohan Li and Dr Yiqun Li for offering advice to the theorem proof. Alain Bensoussan acknowledges the financial support from the National Science Foundation under
grant  DMS-1905449 and grant HKSAR-GRF 14301321.  Jiayue Han acknowledges the support of UGC for PhD candidates.
  Phillip Yam acknowledges the financial supports from HKGRF-14300717 with the project title “\textit{New kinds of Forward-backward Stochastic Systems with Applications}”, HKGRF-14300319 with the project title “\textit{Shape-constrained Inference: Testing for Monotonicity}”, HKGRF-14301321 with the project title “\textit{General Theory for Infinite Dimensional Stochastic Control:
Mean Field and Some Classical Problems}”  and Direct Grant for Research 2014/15 (Project No.\ 4053141) offered by CUHK. 
Xiang Zhou acknowledges the support of Hong Kong RGC GRF grant 11305318.

\bibliographystyle{siamplain}
\bibliography{./ctrl}
\end{document}